\documentclass[a4paper]{amsart}

\usepackage[a4paper,width={16cm},left=2.5cm,bottom=3cm, top=3cm]{geometry}
\usepackage{enumerate}

\usepackage[utf8]{inputenc}
\usepackage[english]{babel}

\usepackage{a4wide}
\usepackage{xcolor}
\usepackage{amsmath}
\usepackage{amssymb}
\usepackage{amsfonts}
\usepackage{amsthm,graphicx}

\usepackage{hyperref}
\usepackage{mathtools}
\mathtoolsset{showonlyrefs}

\newtheorem{theorem}{Theorem}[section]

\newtheorem{corollary}[theorem]{Corollary}

\newtheorem{proposition}[theorem]{Proposition}
\theoremstyle{definition}  
\newtheorem{definition} [theorem] {Definition}

\newtheorem{remark} [theorem] {Remark}

\newcommand{\la}{\lambda}
\newcommand{\norm}[1]{\left\lVert#1\right\rVert}
\newcommand{\pd}[2]{\frac{\partial#1}{\partial#2}}

\newcommand{\R}{\mathbb{R}}

\newcommand{\e}{\varepsilon}
\newcommand{\co}{{\rm{const}}}

\newcommand{\Tr}{\mathop{\rm{Tr}}}
\newcommand{\dive}{\mathop{\rm{div}}}

\newcommand{\df}[4]{\ensuremath\sideset{_{#1}}{_{#4}}{\mathop{\left\langle #2, #3 \right\rangle}}}
\newcommand{\ps}[3]{\left( #2, #3 \right)_{#1}}

\begin{document}

\title[On fractional parabolic equations with Hardy-type potentials]
{On fractional parabolic equations\\ with Hardy-type potentials}
\date{December 10, 2022}

\author{Veronica Felli, Ana Primo and Giovanni Siclari}

\address[V. Felli, G. Siclari]{Dipartimento di Matematica
  e Applicazioni
\newline\indent
Universit\`a degli Studi di Milano - Bicocca
\newline\indent
Via Cozzi 55, 20125, Milano, Italy}
\email{veronica.felli@unimib.it,  g.siclari2@campus.unimib.it}

\address[A. Primo]{Departamento de  Matem{\'a}ticas
\newline\indent
Universidad  Aut\'onoma de  Madrid
\newline\indent 28049, Madrid, Spain
}
\email{ana.primo@uam.es}

\thanks{ 
  V. Felli and G. Siclari are partially
supported by the GNAMPA-INdAM 2022 grant ``Questioni di esistenza e
unicit\`a per problemi non locali con potenziali''. \\ A. Primo is supported by PDI2019-110712GB-I00 project and  ``Regional Programme of Research and Technological Innovation.V PRICIT''. \\
Special thanks to  Departamento de  Matem{\'a}ticas of Universidad Autónoma de Madrid for the hospitality  offered to G. Siclari in July 2022.
}

 \begin{abstract}
   \noindent  
A classification of local asymptotic profiles and strong unique
continuation properties are established for a class of fractional heat
equations with a Hardy-type potential, via an Almgren-Poon monotonicity
formula combined with a blow-up analysis.
\end{abstract}

\maketitle

\section{Introduction} \label{sec-introduction}
This paper  deals with the following singular fractional evolution equation 
\begin{equation}\label{eq-frac-heat-hardy}
(w_t-\Delta w)^s=\frac{1}{\kappa_s}\left(\frac{\mu}{|x|^{2s}} w+g
  w\right),
\quad \text{ in } \R^N \times (t_0-T,t_0),
\end{equation}
where  $T>0$, and  
\begin{equation}\label{s-N-def-lambdaNs-kappas-mu}
s \in (0,1), \quad N>2s, \quad  \mu <\kappa_s\Lambda_{N,s}, \quad \kappa_s:=\frac{\Gamma(1-s)}{2^{2s-1} \Gamma(s)},	 \quad \Lambda_{N,s}:=2^{2s} \frac{\Gamma^2\left(\frac{N+2s}{4}\right)}{\Gamma^2\left(\frac{N-2s}{4}\right)}.
\end{equation} 
We are interested in studying the asymptotic behaviour of solutions to
\eqref{eq-frac-heat-hardy} at $(x,t)=(0,t_0)$ along the directions
$(\la x, t_0-\la^2t)$ as $\la \to 0^+$.  Our main result is a
  classification of possible limiting asymptotic rates and profiles
in terms of the eigenfunctions of a weighted Ornstein-Uhlenbeck
operator, see Theorem \ref{theorem-blow-up} in Section
\ref{sec-functional-settings} for a more precise statement. As a
corollary, we obtain a strong space-like unique continuation property from
the point $(0,t_0)$, see Corollary \ref{corollary-unique-principle} in
Section~\ref{sec-functional-settings}.

On the perturbing potential $g$  we assume the following hypotheses: 
\begin{align}
  &g,g_t\in L^r((t_0-T,t_0),L^{\frac{N}{2s}}(\R^N)), \quad
    g_t\in L^\infty_{loc}((t_0-T,t_0), L^{\frac{N}{2s}}(\R^N)), \label{hp-g-Lr} \\
  &|g(x,t)|+|\nabla g(x,t)\cdot x|\le C_g(1+|x|^{-2s+\varepsilon}) 
    \quad  \text{ for all } t \in (t_0-T,t_0) \text{ and  a.e. } x \in \R^N,  \label{hp-g-subhomogeneous} 
\end{align}
for some constant $C_g>0$, $\varepsilon \in (0,2s)$ and $r>1$. The
gradient of $g$ in \eqref{hp-g-subhomogeneous} is with respect to the
variable $x$ and it is meant in a distributional sense.  The
assumption \eqref{hp-g-subhomogeneous} on the potential $g$ is quite 
natural when dealing with Hardy-type singularities; it 
  essentially says that $g$ is negligible at $0$ with respect to
  the singular potential $\frac{\mu}{|x|^{2s}}$. We recall that a similar assumption is made in \cite{Fall-Felli-2014} for an
analogous equation with the fractional Laplacian $(-\Delta)^s$,
corresponding to
the stationary case of~\eqref{eq-frac-heat-hardy}.

In the literature one may find many definitions of the operator
$H^s (w):=(w_t-\Delta w)^s$ in \eqref{eq-frac-heat-hardy}, that is of
the fractional power of the classical heat operator
$H (w):=w_t-\Delta w$. We refer to \cite{AT} and \cite{ST} for a
presentation of the several ways to define $H^s$ corresponding to
different functional settings. It is also worth mentioning that a
pointwise formula for $H^su$ is derived in
\cite{ST}.  In Section \ref{sec-functional-settings} we give a precise
definition of $H^s$ and of weak solutions to
\eqref{eq-frac-heat-hardy} by the Fourier transform.

Our approach is based on an Almgren-Poon type monotonicity formula,
see \cite{PC}, combined with a blow-up argument. We mention that
monotonicity methods and blow-up analysis are used in \cite{FP} to
prove strong unique continuation and classification of blow-up
profiles for parabolic equations with a Hardy potential (corresponding
to the case $s=1$ in \eqref{eq-frac-heat-hardy}); analogous results
are obtained in \cite{FP2} for a class of parabolic equations with
critical electromagnetic potentials.
    
    There exists a large literature dealing with strong continuation
properties in the local parabolic setting. We mention \cite{LI1} for
unique continuation for parabolic operators with $L^{\frac{N+1}{2}}$ time-independent
coefficients and \cite{SS,SO} for unique continuation on horizontal
components, proved by Carleman weighted inequalities, in the presence of
time-dependent coefficients. The paper \cite{CH} contains not only a
unique continuation result but also some local asymptotic analysis of
solutions to parabolic inequalities with bounded coefficients.
We quote \cite{AV,ES,EF,EKPV,EV,FE} for
unique continuation results for parabolic equations with
time-dependent potentials by Carleman inequalities and monotonicity
methods. We also refer to  \cite{BGM}  for unique
  continuation properties  for the heat operator with a Hardy
  potential established by Carleman estimates.

To deal with the
fractional case \eqref{eq-frac-heat-hardy}, we introduce an
Almgren-Poon frequency function for an equivalent localized problem,
constructed by the extension procedure developed in \cite{BG,BCS,NS,ST}, in the
spirit of the one introduced by Caffarelli and Silvestre  in
\cite{CS} for the
fractional powers of the Laplacian. This leads us to
  deal with equation \eqref{prob-frac-extended-forward}, which is  a local degenerate or singular parabolic problem in
a one more
dimension, see Section \ref{sec-functional-settings} for the 
details.

In a fractional parabolic setting, an Almgren-Poon frequency formula is
first established in \cite{ST} in the absence of potentials, i.e. for
$g\equiv0$ and $\mu=0$. Subsequently, an Almgren-Poon monotonicity
approach is used in \cite{BG} to prove unique continuation properties
for weak solutions to \eqref{eq-frac-heat-hardy} in the case $\mu=0$,
that is without the Hardy singularity, and under $C^1$ or $C^2$
regularity assumptions on the potential $g$, depending on the value of
$s$.  In \cite{BG} a crucial role is played by a H\"older regularity
theory for solutions to the extended problem, which has, in addition
to its independent interest, applications to the estimates needed to
derive an Almgren-Poon type monotonicity formula.  We mention that a
space-like strong unique continuation property is established in
\cite{ABDG} in the case $\mu=0$ via a conditional elliptic type
doubling property and blow-up analysis. The case $\mu=0$ is treated
also in \cite{AT}, where, under similar regularity assumptions on the
potential $g$, a fine analysis of the structure of the nodal set and
of possible blow-ups of solutions vanishing with a finite order is
performed.  The approach of \cite{AT} is also based on an Almgren-Poon
type monotonicity formula and makes use of some uniform H\"older
bounds, improving the regularity estimates of \cite{BG} and providing
an independent proof of the H\"older regularity of weak solutions.

Due to the presence of a Hardy-type potential, there is no hope
to obtain similar regularity results, since weak solutions to
\eqref{eq-frac-heat-hardy} may in general be not bounded, see Theorem
\ref{theorem-blow-up}.  In the spirit of \cite{FP}, to overcome
  this difficulty we
  rely instead on the theory of abstract parabolic equations, once 
a formulation of the extension problem in a suitable Gaussian
space is obtained. 
Furthermore we also obtain a classification
of the asymptotic profiles of weak solutions to
\eqref{eq-frac-heat-hardy} at $(x,t)=(0,t_0)$ along the directions
$(\la x, t_0-\la^2t)$ as $\la \to 0^+$, see Theorem
\ref{theorem-blow-up} and Theorem \ref{theorem-blow-up-extended} in
Section \ref{sec-functional-settings}.

This paper is organized as follows. In Section
\ref{sec-functional-settings} we introduce a functional setting
suitable for the weak formulation of \eqref{eq-frac-heat-hardy}
and state the main results.  In Section \ref{sec-Inequalities-Traces}
we prove some functional inequalities and trace results in Gaussian
spaces. In Section \ref{sec-formulation} we give an alternative
   weak formulation of the extended problem in Gaussian spaces and prove a
regularity result. In Section \ref{sec-eigenvalues} we describe the
eigenvalues of a weighted Ornstein-Uhlenbeck operator which turn
out to be related to the classification of the asymptotic behaviour of
weak solutions to \eqref{eq-frac-heat-hardy} at $(0,t_0)$.  In Section
\ref{sec-Almgren-Poon-type-monotonicity-formula} we derive an
Almgren-Poon type monotonicity formula for the extended problem,
which is combined with a blow-up analysis in Section \ref{sec-blow-up} to
obtain our main results, i.e. the asymptotic of solutions and the strong space-like unique
continuation property.

\section{Functional settings and main results} \label{sec-functional-settings}

To formally introduce the fractional heat operator, let us first set
some notations. For any real Hilbert space $X$ we denote with $X^*$ its
dual space and with $\df{X^*}{\cdot}{\cdot}{X}$ the duality between $X^*$
and $X$; $\ps{X}{\cdot}{\cdot}$ denotes the scalar product in $X$.

The operator  $H^s$ can  be defined by  means of the Fourier transform
as follows: for any function $w \in \mathcal{S}(\R^{N+1})$,  
\begin{equation*}
\widehat{ H^s (w)}(\xi,\theta):=(i\theta+|\xi|^2)^s\widehat{w}(\xi,\theta),
\end{equation*} 
where the Fourier transform of $w$ is defined as
\begin{equation*}
\mathcal{F}(w)(\xi,\theta)=\widehat{w}(\xi,\theta):
=\frac{1}{(2 \pi)^{\frac{N+1}{2}}}\int_{\R^{N+1}}e^{-i(x \cdot\xi+t \theta )}w(x,t) \, dx\,dt.
\end{equation*}
Furthermore, we can extend $H^s$ to its natural domain; more
precisely, we can define $H^s$ on
\begin{equation*}
  \mathop{\rm{Dom}}(H^s):=\left\{w \in L^2(\R^{N+1}):
    \int_{\R^{N+1}}|i\theta+|\xi|^2|^s|\widehat{w}(\xi,\theta)|^2 \,
    d\xi\, d\theta
    <+\infty\right\},
\end{equation*}
endowed with the norm 
\begin{equation*}
\norm{w}_{\mathop{\rm{Dom}}(H^s)}:=\left(\int_{\R^{N+1}} w^2(x,t) \,
  dx\,dt
  +\int_{\R^{N+1}}|i\theta+|\xi|^2|^{s}|\widehat{w}(\xi,\theta)|^2
  d\xi \,
  d\theta\right)^{\frac12},
\end{equation*}
as the map from $\mathop{\rm{Dom}}(H^s)$ into its dual space
$(\mathop{\rm{Dom}}(H^s))^*$, defined as
\begin{equation}\label{def-operator-Hs}
  \df{(\mathop{\rm{Dom}}(H^s))^*}{H^s(w)}{v}{\mathop{\rm{Dom}}(H^s)}:
  =\int_{\R^{N+1}}(i\theta+|\xi|^2)^s
  \widehat{w}(\xi,\theta)\overline{\widehat{v}(\xi,\theta)} d \xi d\eta,
\end{equation} 
for any $w, v \in \mathop{\rm{Dom}}(H^s)$. 

It is worth noticing that, since  $|\xi|^{2s}\le |i\theta+|\xi|^2|^{s}$ for any $(\theta,\xi) \in \R^{N+1}$, 
\begin{equation*}
\norm{v}_{ L^2(\R,W^{s,2}(\R^N)) }\le \norm{v}_{\mathop{\rm{Dom}}(H^s)}
\end{equation*}
for any   $v \in  \mathop{\rm{Dom}}(H^s)$. Hence the natural embedding 
\begin{equation}\label{domHs-inclusion-L2Hs}
\mathop{\rm{Dom}}(H^s) \hookrightarrow L^2(\R,W^{s,2}(\R^N))
\end{equation}
is linear and continuous. With  $W^{s,2}(\R^N)$ we are denoting  the usual fractional Sobolev space.
Furthermore, since we are dealing with  a Hardy-type potential, the weighted $L^2$-space
\begin{equation*}
L^2(\R^N,|x|^{-2s}):=\left\{ v:\R^N\to \R \text{ measurable:} \int_{\R^N}\frac{v^2}{|x|^{2s}} \, dx < + \infty \right\}
\end{equation*}
will play a role in our analysis, together with the following
Hardy-type inequality due to Herbst \cite{HI}:
\begin{equation}\label{ineq-hardy-frac-Herbst}
\int_{\R^N} \,|\xi|^{2s} |\widehat{\phi}|^2\,d\xi\geq
\Lambda_{N,s}\,\int_{\R^N} |x|^{-2s} \phi^2\,dx
\end{equation}
for all $\phi\in
\mathcal{C}_0^\infty(\R^N)$,  where  $\Lambda_{N,s} >0$, defined in \eqref{s-N-def-lambdaNs-kappas-mu}, is optimal and not attained.

In view of \eqref{def-operator-Hs}, we define  a weak  solution of \eqref{eq-frac-heat-hardy} as a function  $w \in \mathop{\rm{Dom}}(H^s)$ such that 
\begin{equation}\label{eq-frac-heat-hardy-weak-forward}
\sideset{_{(\mathop{\rm{Dom}}(H^s))^*}}{_{\mathop{\rm{Dom}}(H^s)}}{\mathop{\left\langle H^s(w), \phi\right\rangle}}=\frac{1}{\kappa_s}\int_{t_0-T}^{t_0}\left(\int_{\R^N}\left(\frac{\mu}{|x|^{2s}} w\phi +gw\phi \right)\, dx \right) dt,
\end{equation}
for any $\phi \in C^\infty_c(\R^N\times(t_0-T,t_0) )$.  In view of \eqref{hp-g-subhomogeneous}, \eqref{domHs-inclusion-L2Hs}, \eqref{ineq-hardy-frac-Herbst},   and the H\"older inequality,  the above  definition of weak solution is well-posed, that is   the right hand side, as a function of $\phi$, belongs to  $({\rm{Dom}}(H^s))^*$.

In order to  develop an Almgren-Poon type monotonicity formula, we
apply the extension procedure of
  \cite{BCS} (see also \cite{BG,NS,ST}) to localize the problem.

Let us set some notation first.
We denote as $(z,t)=(x,y,t)$ the variable in   $\R^{N} \times
  (0,+\infty) \times \R$. Moreover we let
\begin{align*}
&\R^{N+1}_+:=\R^{N} \times (0,+\infty),\\
&B_r^+:=\{z \in \R^{N+1}_+:|z| <r\}, \\
&B_r':=\{(x,0) \in \R^{N+1}:|x| <r\},\\
&S_r^+:=\{z \in \R^{N+1}_+:|z|=r\}, 
\end{align*}
for any $r>0$. With a slight abuse, we often identify $B_r'$ with
$\{x \in \R^{N}:|x|<r\}$.  Furthermore we use the symbols
$\nabla $ and $\dive$ to  denote the gradient, respectively the
divergence, with respect to the space variable $z=(x,y)$.

For any $p \in [1,\infty)$ and any open set $E \subset \R^{N+1}_+$, let   
\begin{equation*}
L^p(E,y^{1-2s}):=\left\{V:E\to \R \text{ measurable }: \int_{E} y^{1-2s}|V|^p \, dz<+\infty\right\}.
\end{equation*}
If $E$ is an  open Lipschitz set contained in $\R^{N+1}_+$,   $H^1(E,y^{1-2s})$ is defined as  the completion of $C_c^{\infty}(\overline{E})$ with  respect to the norm  
\begin{equation*}
\norm{\phi}_{H^1(E,y^{1-2s})}:=\left(\int_{E} y^{1-2s}(\phi^2 +|\nabla \phi|^2)\, dz\right)^{\frac12}.
\end{equation*}
In view of   \cite[Theorem 11.11, Theorem 11.2, 11.12
Remarks(iii)]{KA}  and the extension theorems for weighted Sobolev
spaces with weights in the Muckenhoupt's $A_2$ class proved in
\cite{CSK},  the space $H^1(E,y^{1-2s})$ admits a concrete characterization   as 
\begin{equation*}
H^1(E,y^{1-2s})=\left\{V \in W^{1,1}_{loc}(E):\int_{E} y^{1-2s} (V^2+|\nabla V|^2)\, dz< +\infty\right\}.
\end{equation*}
We also note that by  \cite{LM} there exists a linear and continuous trace operator 
\begin{equation}\label{trace-Hs}
\Tr:H^1(\R_+^{N+1},y^{1-2s}) \to W^{2,s}(\R^{N}).
\end{equation}

The following theorem is a particular case of a very general extension
result proved in \cite{BCS}.  See also \cite[Theorem
]{NS},
\cite[Theorem 1.7]{ST} and \cite[Section 3, Section 4]{BG}.
\begin{theorem}\cite[Theorem 4.1, Remark 4.3]{BCS} \label{theorem-extension}
  If $ w \in \mathop{\rm{Dom}}(H^s)$, then there exists a function
  $W \in L^2(\R,H^1(\R^{N+1}_+,y^{1-2s}))$ that weakly solves
\begin{equation*}
\begin{cases}
  y^{1-2s}W_t-\dive(y^{1-2s}\nabla W)=0, &\text{ in } \R_+^{N+1} \times \R,\\
  \Tr(W(\cdot,t))=w(\cdot,t), &\text{ on }\R^N, \text{ for a.e. } t \in \R,\\
  -\lim\limits_{y \to 0^+}y^{1-2s}\pd{W}{y}= \kappa_s H^s(w), &\text{
    on }\R^N\times \R,
\end{cases}
\end{equation*}
 in the sense that 
\begin{multline}\label{eq-frac-extended-without-eq}
\int_{\R}\left(\int_{\R_+^{N+1}}y^{1-2s} W\phi_t  \, dz \, \right)dt\\
=\int_{\R}\left(\int_{\R_+^{N+1}}y^{1-2s} \nabla W \cdot \nabla \phi \, dz \, \right)dt-\kappa_s\df{(\mathop{\rm{Dom}}(H^s))^*}{H^s(w)}{\phi(\cdot,0,\cdot)}{\mathop{\rm{Dom}}(H^s)} \notag 
\end{multline}
for any $\phi \in C^\infty_c (\overline{\R_+^{N+1}} \times \R)$.	
\end{theorem}

The following corollary is an easy consequence of Theorem \ref{theorem-extension} and \eqref{eq-frac-heat-hardy-weak-forward}.
\begin{corollary}\label{corollary-frac-extended-forward}
If $ w \in \mathop{\rm{Dom}}(H^s)$  solves \eqref{eq-frac-heat-hardy-weak-forward},  then there exists $W \in L^2(\R,H^1(\R^{N+1}_+,y^{1-2s}))$ 
that weakly solves 
\begin{equation}\label{prob-frac-extended-forward}
\begin{cases}
y^{1-2s}W_t-\dive(y^{1-2s}\nabla W)=0, &\text{ in } \R_+^{N+1} \times (t_0-T,t_0),\\
\Tr( W(\cdot,t))= w(\cdot,t), &\text{ on }\R^N, \text{ for a.e. } t \in  (t_0-T,t_0),\\
-\lim\limits_{y \to 0^+}y^{1-2s}\pd{W}{y}= \frac{\mu}{|x|^{2s}} w +g w, &\text{ on }\R^N\times (t_0-T,t_0),
\end{cases}
\end{equation}
in the sense that 
\begin{multline}\label{eq-weak-frac-extended-forward}
\int_{t_0-T}^{t_0}\left(\int_{\R_+^{N+1}}y^{1-2s} W\phi_t  \, dz \, \right)dt\\
=\int_{t_0-T}^{t_0}\left(\int_{\R_+^{N+1}}y^{1-2s}  \nabla  W \cdot \nabla \phi \, dz \, \right)dt-\int_{t_0-T}^{t_0}\left(\int_{\R^N}\left(\frac{\mu}{|x|^{2s}} w \phi +g w\phi \right)\, dx \right) dt, 
\end{multline}
for any $\phi \in C^\infty_c (\overline{\R_+^{N+1}} \times(t_0-T,t_0))$.	
\end{corollary}
The asymptotic behavior at $(0,t_0)$ of a solution $W$ of \eqref{prob-frac-extended-forward}, and consequently  of a solution $w$ of \eqref{eq-frac-heat-hardy-weak-forward}, will turn out to be related to the following eigenvalue problem for a weighted  Ornstein-Uhlenbeck operator: 
\begin{equation}\label{prob-eigenvalue-Ornstein-Uhlenbeck-operator}
\begin{cases}
-\dive(y^{1-2s}\nabla Y)+y^{1-2s}\frac{z}{2}\cdot \nabla Y=\gamma y^{1-2s}Y, &\text{ in } \R^{N+1}_+,\\
-\lim\limits_{y \to 0^+}y^{1-2s}\pd{Y}{y}=\frac{\mu}{|x|^{2s}}\Tr(Y), &\text{ on } \R^{N},
\end{cases}
\end{equation}
with $\mu <\kappa_s\Lambda_{N,s}$, see
\eqref{s-N-def-lambdaNs-kappas-mu} for the definition of $\kappa_s$
and $\Lambda_{N,s}$.  To introduce a suitable functional setting for
problem \eqref{prob-eigenvalue-Ornstein-Uhlenbeck-operator}, we define
\begin{equation*}
  G_s(z,t):=	t^{-\frac{N+2-2s}{2}} e^{-\frac{|z|^2}{4t}} \quad
  \text{ for any } (z,t) \in \R^{N+1}_+ \times (0,\infty).
\end{equation*}
It is easy to verify that $G_s \in C^{\infty}(\R^{N+1}_+ \times (0,\infty))$ solves the problem 
\begin{equation*}
\begin{cases}
y^{1-2s}\pd{G_s}{t} -\dive(y^{1-2s}\nabla G_s)=0, &\text{ in } \R^{N+1}_+ \times (0,\infty),\\
\lim\limits_{y \to 0^+} y^{1-2s} \pd{G_s}{y}=0, & \text{ on }\R^{N}  \times (0,\infty),
\end{cases}
\end{equation*}
in a classical sense.
Furthermore 
\begin{equation}\label{eq-G-nabla}
  \nabla G_s(z,t)=-\frac{z}{2t} G_s(z,t) \quad \text{ for any }(z,t) \in \R^{N+1}_+ \times (0,\infty).
\end{equation}
Letting 
\begin{equation*}
  G(z):=G_{s}(z,1)=e^{-\frac{|z|^2}{4}}, \quad \text{ for any } z \in \R^{N+1},
\end{equation*}
we define 
\begin{equation*}
  \mathcal{L}:=\left\{V:\R^{N+1}_+\to \R \text{ measurable}: \int_{\R_+^{N+1}} y^{1-2s}V^2    G \, dz  <+\infty\right\}
\end{equation*}
and  $\mathcal{H}$ as the completion  of $C^{\infty}_c(\overline{\R_+^{N+1}})$ with respect to the norm 
\begin{equation*}
  \norm{\phi}_{\mathcal{H}}:=\left( \int_{\R_+^{N+1}} y^{1-2s}(\phi^2 +|\nabla \phi |^2) G\, dz\right)^\frac12.
\end{equation*}
It is clear that both $\mathcal{L}$ and $\mathcal{H}$ are Hilbert
spaces with respect to the natural scalar product associated to the
$\norm{\,\cdot\,}_{\mathcal{L}}$-norm and the
$\norm{\,\cdot\,}_{\mathcal{H}}$-norm respectively.  We observe that
\begin{equation*}
\norm{W}_{\mathcal{H}} \le  \norm{W}_{H^1(\R^{N+1}_+,y^{1-2s})} \quad \text{ for any } W \in H^1(\R^{N+1}_+,y^{1-2s}),
\end{equation*}
hence the embedding
\begin{equation*}
H^1(\R^{N+1}_+,y^{1-2s})\hookrightarrow \mathcal{H}
\end{equation*}
is linear and continuous.
Furthermore, we consider the weighted $L^2$-spaces 
\begin{equation*}
L^2(\R^N,G(x,0)):=\left\{v:\R^{N}\to \R \text{ measurable}: \int_{\R^N}v^2(x)  G(x,0) \, dx  <+\infty\right\}
\end{equation*}
and 
\begin{equation*}
L^2(\R^N, |x|^{-2s} G(x,0)):=\left\{v:\R^{N}\to \R \text{ measurable}: \int_{\R^N} \frac{v^2(x)}{|x|^{2s}}    G(x,0) \, dx  <+\infty\right\}.
\end{equation*}
The trace operator $\Tr$ introduced in \eqref{trace-Hs} can be extended to a continuous linear trace operator, still denoted as $\Tr$, from $\mathcal{H}$ to $L^2(\R^N,G(x,0))$, see Proposition \ref{prop-trace-ineq} in Section \ref{sec-Inequalities-Traces}.
Furthermore $\Tr$ takes values in $L^2(\R^N,|x|^{-2s}G(x,0))$ and 
\begin{equation*}
\Tr: \mathcal{H}\to L^2(\R^N,|x|^{-2s}G(x,0)),
\end{equation*}
is  linear and continuous, see Proposition \ref{prop-trace-L-2-x-2s} in  Section \ref{sec-Inequalities-Traces}.

We say that $\gamma$ is an eigenvalue of problem
  \eqref{prob-eigenvalue-Ornstein-Uhlenbeck-operator} if there exists
  an eigenfunction $Y\in \mathcal{H}\setminus\{0\}$ weakly satisfying
  \eqref{prob-eigenvalue-Ornstein-Uhlenbeck-operator}, i.e. 
\begin{equation}\label{eq-eigenvalue-Ornstein-Uhlenbeck-operator}
\int_{\R^{N+1}_+}y^{1-2s} \nabla Y \cdot\nabla V \, G\, dz-\int_{\R^N} \frac{\mu}{|x|^{2s}}\Tr(Y)\Tr(V)   \, G(0,\cdot)\, dx
=\gamma \int_{\R^{N+1}_+}y^{1-2s}  Y V \,  G\, dz
\end{equation} 
for any $V \in \mathcal{H}$. Proposition \ref{prop-trace-L-2-x-2s} in
Section \ref{sec-Inequalities-Traces} ensures that the above
definition of weak solution is well posed.

In order to compute the eigenvalues of
\eqref{prob-eigenvalue-Ornstein-Uhlenbeck-operator} we separate the
variable $z$ in radial and angular parts. Henceforward we denote
\begin{equation*}
\mathbb{S}^{N-1}:=\{x \in \R^{N}: |x|=1\},\quad
\mathbb{S}_+^{N}:=\{z \in \R_+^{N+1}: |z|=1\},
\end{equation*}
identifying  $\partial \mathbb{S}^N_+$ with $\mathbb{S}^{N-1}$.
Writing  as $\theta=(\theta_1,\dots, \theta_{N+1})$ the coordinates on
$\mathbb{S}^{N}$, we define
\begin{equation*}
L^2	(\mathbb{S}^N_+,\theta_{N+1}^{1-2s}):=\left\{v:\mathbb{S}^N_+ \to \R \text{ measurable}: \int_{\mathbb{S}^N_+} \theta_{N+1}^{1-2s}|v|^2 \, dz<+\infty\right\}
\end{equation*}
and $H^1(\mathbb{S}^N_+,\theta_{N+1}^{1-2s})$ as the completion of
$C^{\infty}_c(\overline{\mathbb{S}^N_+})$ with respect to the norm

\begin{equation*}
  \norm{\phi}_{H^1(\mathbb{S}^N_+,\theta_{N+1}^{1-2s})}:=
  \left(\int_{\mathbb{S}^N_+} \theta_{N+1}^{1-2s}(|\phi|^2
    +|\nabla_{\mathbb{S}^N} \phi|^2) \, dS\right)^{\frac{1}{2}},
\end{equation*}
where $\nabla_{\mathbb{S}^N}$ and $dS$ denote the Riemannian
gradient 
and the volume element, respectively,
with respect to the standard metric on the unit $N$-dimensional
  sphere $\mathbb{S}^N$.

We refer to \cite{Fall-Felli-2014} for the following proposition.
\begin{proposition}\cite [Lemma 2.2]{Fall-Felli-2014}
There exists a linear and continuous trace operator 
\begin{equation}\label{def-trace-S}
  \mathcal{T}: H^1(\mathbb{S}^N_+,\theta_{N+1}^{1-2s})
  \to L^2(\mathbb{S}^{N-1})=L^2(\partial \mathbb{S}^N_+).
\end{equation}
Furthermore, letting $\kappa_s$ and $\Lambda_{N,s}$ be as in
\eqref{s-N-def-lambdaNs-kappas-mu},
\begin{equation}\label{ineq-trace-S}
  \kappa_s\Lambda_{N,s} \int_{\mathbb{S}^{N-1}} |\mathcal{T}(V)|^2 dS'
  \le \left(\frac{N-2s}{2}\right)^2 \int_{\mathbb{S}_+^N}\theta_{N+1}^{1-2s} |V|^2 dS
  +\int_{\mathbb{S}_+^N}\theta_{N+1}^{1-2s} |\nabla_{\mathbb{S}^N}V|^2 dS
\end{equation}
for any $v \in H^1(\mathbb{S}^N, \theta_{N+1}^{1-2s})$, where
  $dS'$ denotes the volume element on $\mathbb{S}^{N-1}$.
\end{proposition}
Let us consider the following eigenvalue problem 
\begin{equation}\label{prob-eigenvalue-half-a-sphere}
\begin{cases}
  -\dive_{\mathbb{S}^{N}}(\theta_{N+1}^{1-2s}\nabla_{\mathbb{S}^{N}}
  \psi)
  =\nu \theta_{N+1}^{1-2s}\psi, &\text{ in } \mathbb{S}^{N}_+,\\
  -\lim\limits_{\theta_{N+1} \to
    0^+}\theta_{N+1}^{1-2s}\nabla_{\mathbb{S}^{N}}\psi \cdot
  e_{N+1}= \mu\mathcal{T}(\psi), &\text{ on } \mathbb{S}^{N-1},
\end{cases}
\end{equation}
where $e_{N+1}:=(0, \dots,1) \in \R^{N+1}$ and
$\mu<\kappa_s \Lambda_{N,s}$ as in
\eqref{s-N-def-lambdaNs-kappas-mu}. We say that $\nu\in \R$ is an
eigenvalue of \eqref{prob-eigenvalue-half-a-sphere} if there exists
$\psi \in H^1(\mathbb{S}^N_+,\theta_{N+1}^{1-2s})\setminus \{0\}$,
called eigenfunction, such that
\begin{equation*}
  \int_{\mathbb{S}^N_+} \theta_{N+1}^{1-2s}\nabla_{\mathbb{S}^N} \psi
  \cdot \nabla_{\mathbb{S}^N} V\, dS
  -\mu \int_{\mathbb{S}^{N-1}} \mathcal{T}(\psi) \mathcal{T}(V) \, dS'
  =\nu \int_{\mathbb{S}^N_+} \theta_{N+1}^{1-2s} \psi V \, dS
\end{equation*}
for any $V \in H^1(\mathbb{S}^N_+,\theta_{N+1}^{1-2s})$.  Since the
natural embedding
$H^1(\mathbb{S}^N_+,\theta_{N+1}^{1-2s})\hookrightarrow
L^2(\mathbb{S}^N_+,\theta_{N+1}^{1-2s})$ is compact, see \cite{FKCSR}
and \cite{Opic-Kufner},
by classical spectral theory the eigenvalues of
\eqref{prob-eigenvalue-half-a-sphere} are a non-decreasing and
diverging sequence
$\{\nu_k(\mu)\}_{k \in \mathbb{N} \setminus \{0\}}$. In the sequence
$\{\nu_k(\mu)\}_{k \in \mathbb{N}\setminus \{0\}}$ we repeat each
eigenvalue as many times as the dimension of the associated
eigenspace.  Inequality \eqref{ineq-trace-S} implies the following
estimate on the first eigenvalue:
\begin{equation}\label{ineq-nu-1}
	\nu_1(\mu) >-\left(\frac{N-2s}{2}\right)^2.
\end{equation}
Furthermore there exists an orthonormal basis
$\{\psi_k\}_{k \in \mathbb{N} \setminus \{0\}} $ of
$L^2(\mathbb{S}^N_+,\theta_{N+1}^{1-2s})$ such that, for any
$k \in \mathbb{N}\setminus\{0\}$, the function $\psi_k$ is an
eigenfunction of problem \eqref{prob-eigenvalue-half-a-sphere}
associated to $\nu_k(\mu)$.

\begin{remark} If $\mu=0$, then a
  combination of the regularity result of \cite[Theorem 1.1]{STV} with
  the blow-up analysis done in \cite{Fall-Felli-2014} for the
  Caffarelli-Silvestre extended problem implies that the set of
  eigenvalues of \eqref{prob-eigenvalue-half-a-sphere} is
  $\{k^2+k(N-2s):k \in \mathbb{N}\}$.
\end{remark}

Let, for any $n \in\mathbb{N}$ and $j \in \mathbb{N} \setminus\{0\}$,
\begin{equation}\label{def-egienfunctions}
Y_{n,j}(z):=|z|^{-\alpha_j} P_{j,n}\left(\frac{|z|^2}{4}\right)\psi_j\left(\frac{z}{|z|}\right)
\end{equation}
where 
\begin{align}
  &\alpha_j:=\frac{N-2s}{2}- \sqrt{\left(\frac{N-2s}{2}\right)^2+\nu_j(\mu)},\label{def-alphak}  \\
  &P_{j,n}(t):=\sum_{i=0}^n\frac{(-n)_i}{\left(\frac{N+2-2s}{2}-\alpha_j\right)_i}\frac{t^i}{i!}, \quad \text{ with }\quad 
\begin{cases}
(s)_i=\prod_{j=0}^{i-1}(s+j),\\
(s)_0=1.
\end{cases} \label{def-Pnj}
\end{align}
Let us also consider the $\mathcal L$-normalized functions 
\begin{equation}\label{def-egienfunctions-normalised}
  \widetilde{Y}_{n,j}:=\frac{Y_{n,j}}{\norm{Y_{n,j}}_{\mathcal{L}}}
  \quad \text{ for any } (n,j) \in \mathbb{N} \times \mathbb{N} \setminus \{0\}.
\end{equation}
The following result is proved in Section \ref{sec-eigenvalues}
  and provides a complete description of the spectrum of problem
  \eqref{prob-eigenvalue-Ornstein-Uhlenbeck-operator}.

\begin{proposition}\label{prop-eigenvalues}
The set of eigenvalues of problem \eqref{prob-eigenvalue-Ornstein-Uhlenbeck-operator} is
\begin{equation}\label{def-eigenvalues}
  \left\{\gamma_{m,k}:=m -\frac{\alpha_k}{2}:
    k\in \mathbb{N} \setminus \{0\}, m \in \mathbb{N}\right\},
\end{equation} 
where $\{\nu_k(\mu)\}_{k\in \mathbb{N} \setminus \{0\}}$ are the
eigenvalues of problem \eqref{prob-eigenvalue-half-a-sphere} and
$\alpha_k$ is defined in \eqref{def-alphak}.  The multiplicity of each
eigenvalue $\gamma_{m,k}$ is finite and equal to
\begin{equation*}
\#\left\{j \in \mathbb{N}\setminus \{0\}:\gamma_{m,k}+\frac{\alpha_j}{2} \in \mathbb{N}\right\}.
\end{equation*} 
Furthermore, for any $(m,k) \in \mathbb{N} \times \mathbb{N}\setminus \{0\}$,
\begin{equation*}
	E_{m,k}=\left\{\widetilde{Y}_{n,j}:(n,j) \in \mathbb{N} \times\mathbb{N}\setminus \{0\} \text{ and }  \gamma_{m,k}=n - \frac{\alpha_j}{2}\right\}
\end{equation*} 
is an $\mathcal{L}$-orthonormal basis of the eigenspace associated to the eigenvalue $\gamma_{m,k}$, where $\widetilde{Y}_{n,j}$ has been defined in \eqref{def-egienfunctions-normalised}.
Finally 
\begin{equation}\label{def-basis-egienfunctions}
\bigcup_{(m,k) \in \mathbb{N} \times \mathbb{N}\setminus \{0\}}E_{m,k}
\end{equation} 
is a orthonormal basis of  $\mathcal{L}$.

\end{proposition}

The main result of the present paper is the following
classification of the asymptotic behavior near $(0,t_0)$ of any
solution $W$ of \eqref{prob-frac-extended-forward}, based on the limit
as $t \to t_0^-$ of the following Almgren-Poon type frequency function
\begin{multline}\label{def-N-forward}
  \mathcal{N}(t)\\:=\frac{(t_0-t)\left(\int_{\R^{N+1}_+}
      y^{1-2s}|\nabla W|^2 G_s(z,t_0-t) \, dz-\int_{\R^N}
      \left(\frac{\mu}{|x|^{2s}}w^2 +g w^2\right)G_s(x,0,t_0-t)\,
      dx\right)}{\int_{\R^{N+1}_+} y^{1-2s} W^2 G_s(\cdot,t_0-t) \,
    dz}.
\end{multline}

\begin{theorem}\label{theorem-blow-up-extended}
  Let $W \not \equiv 0$ be a weak solution to
  \eqref{prob-frac-extended-forward}. Then there exist
  $m_0 \in \mathbb{N}$ and $k_0 \in \mathbb{N}\setminus\{0\}$ such
  that
\begin{equation}\label{limit-N-forward}
\lim_{t \to t_0^-} \mathcal{N}(t)=\gamma_{m_0,k_0},		
\end{equation}
where $\mathcal{N}$ has been defined in \eqref{def-N-forward} and $\gamma_{m_0,k_0}$ in \eqref{def-eigenvalues}.
Furthermore, letting 
\begin{equation}\label{def-J0}
J_0:=\left\{(m,k) \in \mathbb{N}\times\mathbb{N}\setminus \{0\}:\gamma_{m_0,k_0}=m-\frac{\alpha_k}{2}\right\},
\end{equation}
for any $\tau \in (0,1)$ 
\begin{equation*}
\lim_{\la \to 0^+}\int_{\tau}^{1}\norm{\la^{-2\gamma_{m_0,k_0}}W(\la z\sqrt{t},t_0-\la^2 t)-t^{\gamma_{m_0,k_0}}\sum_{(m,k) \in J_0}\beta_{m,k}\widetilde{Y}_{m,k}(z)}^2_{\mathcal{H}} dt=0
\end{equation*}
and 
\begin{equation*}
\lim_{\la \to 0^+}\sup_{t \in [\tau,1]}\norm{\la^{-2\gamma_{m_0,k_0}}W(\la z\sqrt{t},t_0-\la^2 t)-t^{\gamma_{m_0,k_0}}\sum_{(m,k) \in J_0}\beta_{m,k}\widetilde{Y}_{m,k}(z)}^2_{\mathcal{L}}=0
\end{equation*}
where $\widetilde{Y}_{m,k}$ has been defined in \eqref{def-egienfunctions-normalised}, 
\begin{multline}\label{def-beta}
  \beta_{m,k}=\Lambda^{-2\gamma_{m_0,k_0}}\int_{\R^{N+1}_+}
  y^{1-2s}W(\Lambda z,t_0-\Lambda^2)
  \widetilde{Y}_{m,k}(z) G(z)\, dz \\
  +2\int^\Lambda_0\tau^{2s-1-2\gamma_{m_0,k_0}}\left(\int_{\R^N}
    g(\tau x,t_0-\tau^2)\Tr(W)(\tau x,t_0-\tau
    ^2)\Tr(\widetilde{Y}_{m,k})(x) G(x,0)\, dx \right)\, d\tau,
\end{multline}
for any $\Lambda \in (0,\Lambda_0)$ and for some $\Lambda_0 \in
(0,\sqrt{T})$, and $\Tr$ has been defined in \eqref{trace-Hs}. Finally  $\beta_{m,k} \neq 0$ for some $(m,k) \in J_0$.
\end{theorem}

From Theorem \ref{theorem-blow-up-extended} and the relationship
between problems 
\eqref{eq-frac-heat-hardy-weak-forward} and \eqref{prob-frac-extended-forward}  given by Corollary
\ref{corollary-frac-extended-forward} we can easily deduce a similar
result for solutions to \eqref{eq-frac-heat-hardy-weak-forward}.
\begin{theorem}\label{theorem-blow-up}
  Let $w \not \equiv 0$ be a solution to
  \eqref{eq-frac-heat-hardy-weak-forward}. Then there exist
  $m_0 \in \mathbb{N}$ and $k_0 \in \mathbb{N}\setminus\{0\}$ such
  that, for any $\tau \in (0,1)$, 
\begin{equation*}
  \lim_{\la \to 0^+}\int_{\tau}^{1}\norm{\la^{-2\gamma_{m_0,k_0}}
    w(\la x\sqrt{t},t_0-\la^2 t)-t^{\gamma_{m_0,k_0}}\sum_{(m,k) \in
      J_0}\beta_{m,k}
    \Tr(\widetilde{Y}_{m,k})(x )}^2_{L^2(\R^N, G(\cdot,0))}dt =0,
\end{equation*}
where $\widetilde{Y}_{m,k}$,  $\beta_{m,k}$, and  $\Tr$ are defined in
\eqref{def-egienfunctions-normalised}, \eqref{def-beta}, and
\eqref{trace-Hs}, respectively.
\end{theorem}

Thanks to Theorem
\ref{theorem-blow-up-extended} and Theorem \ref{theorem-blow-up}, we can prove that a strong unique
continuation principle holds for solutions of equations
\eqref{eq-frac-heat-hardy-weak-forward} and
\eqref{eq-weak-frac-extended-forward}.
\begin{corollary} \label{corollary-unique-principle-extended}
Let $W$ be a weak solution of problem
\eqref{prob-frac-extended-forward} such  that 
\begin{equation}\label{eq-W-O}
  W(z,t)=O\left((|z|^2+(t_0-t))^k\right) \quad
  \text{ as } z \to 0 \text{ and } t \to t_0^- \quad \text{for all } k \in \mathbb{N}.
\end{equation}
Then $W\equiv 0$ on $\R^{N+1}_+\times(t_0-T,t_0)$.
\end{corollary}

\begin{corollary}\label{corollary-unique-principle}
Let $w$ be a solution of \eqref{eq-frac-heat-hardy-weak-forward} such that 
\begin{equation*}
w(x,t)=O\left((|x|^2+(t_0-t))^k\right) \quad \text{ as } x \to 0  \text{ and } t \to t_0^- \quad \text{ for all } k \in \mathbb{N}.
\end{equation*}
Then $w\equiv 0$ in $\R^{N}\times(t_0-T,t_0)$.
\end{corollary}

The next theorem is a backward uniqueness result for the Cauchy
problem associated with \eqref{prob-frac-extended-forward}. Its
  proof relies exclusively on the monotonicity argument developed in Section
\ref{sec-Almgren-Poon-type-monotonicity-formula} and does not require
the blow-up argument which is instead needed to obtain the above space-like unique
continuation properties.

\begin{theorem} \label{theorem-unicity-Cauchy}
If $W$ is a solution of \eqref{prob-frac-extended-forward} and there exists $ t_1 \in (t_0-T,t_0)$ such that 
\begin{equation*}
W\left(z, t_1\right)= 0 \quad \text{ for a.e. } z \in \R^{N+1}_+,
\end{equation*}
then $W \equiv 0$ in $\R^{N+1}_+\times(t_0-T,t_0)$.
\end{theorem}

\section{Inequalities and  Traces in Gaussian spaces}  \label{sec-Inequalities-Traces}
In this section we prove some inequalities and trace results for Gaussian spaces.
We start with a Hardy-type inequality. 
\begin{proposition}\label{prop-ineq-hardy-extended}
For any $V \in \mathcal{H}$ 
\begin{multline}\label{ineq-hardy-extended}
  \int_{\R_+^{N+1}}y^{1-2s} \frac{V^2}{|z|^2} G \, dz +\frac{1}{4(N-2s)^2}\int_{\R_+^{N+1}}y^{1-2s} |z|^2 V^2G \, dz\\
  \le \frac{4}{(N-2s)^2}\int_{\R_+^{N+1}}y^{1-2s} |\nabla V|^2G \,
  dz+\frac{N+2-2s}{(N-2s)^2}\int_{\R_+^{N+1}}y^{1-2s} V^2G \, dz.
\end{multline}
\end{proposition}

\begin{proof}
  By density, it is enough to prove \eqref{ineq-hardy-extended}
  for any $\phi \in C^{\infty}_c(\overline{\R_+^{N+1}})$. Thanks to
  \cite[Lemma 2.4]{Fall-Felli-2014}
\begin{multline}\label{ineq-hardy-extended:1}
  \int_{\R_+^{N+1}}y^{1-2s} \frac{\phi^2}{|z|^2} G\, dz \le
  \frac{4}{(N-2s)^2}\int_{\R^{N+1}}y^{1-2s} \left|\nabla\left( \phi e^{-\frac{|z|^2}{8}}\right)\right|^2 \, dz\\
  =\frac{4}{(N-2s)^2}\int_{\R^{N+1}}y^{1-2s} \left(\left|\nabla
      \phi\right|^2 -\frac{1}{4}\nabla \left(\phi^2\right) \cdot z+
    \frac{1}{16}|z|^2 \phi^2\right)e^{-\frac{|z|^2}{4}}\, dz.
\end{multline}  
Let, for any $\delta >0$,
\begin{equation}\label{R-N-delta}
\R^{N+1}_{\delta}:= \{(x,y)\in \R^{N+1}:y >\delta\}. 
\end{equation}
Since  on $\R^{N+1}_+$
\begin{equation*}
  \dive(y^{1-2s}\phi^2 e^{-\frac{|z|^2}{4}}z)=
  y^{1-2s}\left[(1-2s)\phi^2+\nabla\left(\phi^2\right)\cdot z
    -\frac{1}{2}|z|^2\phi^2+(N+1)\phi^2\right]e^{-\frac{|z|^2}{4}},
\end{equation*}
then 
\begin{multline*}
  \int_{\R_\delta^{N+1}}y^{1-2s}\nabla\left(\phi^2\right)\cdot
  z\,e^{-\frac{|z|^2}{4}} dz
  =\int_{\R_\delta^{N+1}}y^{1-2s}\left[(-N-2+2s)\phi^2+\frac{1}{2}|z|^2\phi^2\right]e^{-\frac{|z|^2}{4}} dz \\
  -\delta^{2-2s}\int_{\R^N} \phi^2(x,\delta)
  e^{-\frac{|x|^2+\delta^2}{4}} \, dx.
\end{multline*}
Since $(2-2s)>0$, we can pass to the limit as $\delta \to 0^+$ and conclude that 
\begin{equation*}
  \int_{\R_+^{N+1}}y^{1-2s}\nabla\left(\phi^2\right)\cdot
  z\,e^{-\frac{|z|^2}{4}} dz
  =\int_{\R_+^{N+1}}y^{1-2s}\left[(-N-2+2s)\phi^2+\frac{1}{2}|z|^2\phi^2\right]e^{-\frac{|z|^2}{4}} dz. 
\end{equation*}
Then from \eqref{ineq-hardy-extended:1} we deduce that 
\begin{multline*}
  \int_{\R_+^{N+1}}y^{1-2s} \frac{\phi^2}{|z|^2} G\, dz \\
  \le \frac{4}{(N-2s)^2}\int_{\R^{N+1}}y^{1-2s} \left(\left|\nabla
      \phi\right|^2 + \frac{1}{16}|z|^2
    \phi^2+\frac{1}{4}(N+2-2s)\phi^2-\frac{1}{8}|z|^2
    \phi^2\right)e^{-\frac{|z|^2}{4}}\, dz.
\end{multline*}  
which proves \eqref{ineq-hardy-extended}.
\end{proof}

\begin{proposition}\label{prop-H-t-embbeded-H1}
Let $V \in \mathcal{H}$. Then $V \sqrt{G} \in H^1(\R^{N+1}_+,y^{1-2s})$ and 
\begin{equation}\label{ineq-VsqrtG}
  \norm{\nabla\left( V \sqrt{G}\right)}_{L^2(\R^{N+1}_+,y^{1-2s})}^2
  \le \frac{(N+2-2s)}{2}\int_{\R_+^{N+1}} y^{1-2s}V^2 G \, dz
  +4 \int_{\R_+^{N+1}} y^{1-2s}|\nabla V |^2 G \, dz.
\end{equation}
\end{proposition}

\begin{proof}
If $ V \in \mathcal{H}$ then, in view of \eqref{eq-G-nabla},
\begin{equation*}
  |\nabla(V \sqrt{G})|^2=\left|\nabla V \sqrt{G}
    +\frac{1}{2}V G^{-\frac{1}{2}}\nabla G \right|^2 \le 2 |\nabla V|^2 G +\frac{1}{8}V^2|z|^2 G
\end{equation*}
and so by Proposition \ref{prop-ineq-hardy-extended} it is a clear
that $V \sqrt{G} \in H^1(\R^{N+1}_+,y^{1-2s})$ and \eqref{ineq-VsqrtG}
holds.
\end{proof}

\begin{proposition}\label{prop-trace-ineq}
  The trace operator $\Tr$ introduced in \eqref{trace-Hs} can be
  extended to a linear and continuous trace operator, still denoted as
  $\Tr$,
\begin{equation}\label{def-trace-H-s}
\Tr:\mathcal{H}\to L^2(\R^N,G(x,0)).
\end{equation}
 In particular there exists a constant $K_{N,s}>0$, which depends only
 on $N$ and $s$, such that, for any $V \in
 \mathcal{H}$,
\begin{equation}\label{ineq-trace}
  \int_{\R^{N}} |\Tr(V)|^2G(\cdot,0) \, dx 
  \le K_{N,s} \left(\int_{\R^{N+1}_+}y^{1-2s}|\nabla V|^2 G\, dz
    + \int_{\R^{N+1}_+}y^{1-2s} V^2  G \, dz\right).
\end{equation}
\end{proposition}

\begin{proof}
  There exists a constant $C_{N,s}>0$, which depends only on $N$ and
  $s$, such that, for any
  $\phi \in C^{\infty}_c(\overline{\R^{N+1}_+})$,
\begin{equation}\label{prop-trace-ineq:2}
  \int_{\R^{N}} |\phi(x,0)|^2\, dx\le C_{N,s}
  \left(\int_{\R^{N+1}_+}y^{1-2s}|\nabla \phi|^2\, dz +
    \int_{\R^{N+1}_+}y^{1-2s} \phi^2\, dz\right),
\end{equation}
see for example \cite{LM}. Testing \eqref{prop-trace-ineq:2} with
$\phi \sqrt{G}\in C^{\infty}_c(\overline{\R^{N+1}_+})$,   by Proposition
\ref{prop-H-t-embbeded-H1} we obtain
\eqref{ineq-trace} for any
$\phi \in C^{\infty}_c(\overline{\R^{N+1}_+})$.  Then the operator $\Tr$ is densely
defined on $\mathcal{H}$ and it is continuous. Hence it can be
extended to a continuous trace operator on $\mathcal{H}$ satisfying
\eqref{ineq-trace}.
\end{proof}

\begin{proposition}\label{prop-ineq-hardy-frac}
  Letting $\kappa_{s}$ and $\Lambda_{N,s}$ be as in
  \eqref{s-N-def-lambdaNs-kappas-mu}, for any function
  $V \in \mathcal{H}$
\begin{multline}\label{ineq-hardy-frac}
\kappa_{s} \Lambda_{N,s}\int_{\R^N}\frac{|\Tr(V)|^2}{|x|^{2s}}
G(\cdot,0) \, dx
+\frac{1}{16}\int_{\R_+^{N+1}}y^{1-2s}|z|^2 V^2 G \, dz\\
\le\int_{\R_+^{N+1}}y^{1-2s}|\nabla V|^2 G \, dz+\frac{N+2-2s}{4}
\int_{\R_+^{N+1}}y^{1-2s} V^2 G \, dz.
\end{multline}
\end{proposition}

\begin{proof}
  It is enough to prove \eqref{ineq-hardy-frac} for any
  $ \phi \in C^{\infty}_c(\overline{\R_+^{N+1}})$. Thanks to
  \cite[Lemma 2.5]{Fall-Felli-2014}, for any
  $ \phi \in C^{\infty}_c(\overline{\R_+^{N+1}})$
\begin{equation*}
  \int_{\R^N}\frac{\phi(\cdot,0)^2}{|x|^{2s}} G \, dx \le
  \kappa_s^{-1} \Lambda_{N,s}^{-1}\int_{\R_+^{N+1}}y^{1-2s}
  \left|\nabla \Big(\phi e^{-\frac{|z|^2}{8}}\Big)\right|^2  \, dz.
\end{equation*}
Then we can follow the proof of Proposition
\ref{prop-ineq-hardy-extended} to conclude that
\eqref{ineq-hardy-frac} holds.
\end{proof}

Proposition \ref{prop-ineq-hardy-frac} directly implies the
  following trace result. 
\begin{proposition}\label{prop-trace-L-2-x-2s}
  Let $\Tr$ be the trace operator introduced in
  \eqref{trace-Hs}.  Then
\begin{equation*}
  \Tr(\mathcal{H}) \subseteq L^2(\R^N,|x|^{-2s}G(x,0))
\end{equation*}
and $\Tr:\mathcal{H} \to L^2(\R^N,|x|^{-2s}G(x,0))$ 
is a well defined, linear and continuous  operator.
\end{proposition}

\begin{proposition}\label{prop-equiv-norm}
  For $\mu$ being as in \eqref{s-N-def-lambdaNs-kappas-mu}, let us
  consider the quadratic form
\begin{equation*}
  B(V):=\int_{\R_+^{N+1}}y^{1-2s}|\nabla V|^2 G \, dz-
  \mu \int_{\R^N}\frac{|\Tr(V)|^2}{|x|^{2s}} G(\cdot,0) \, dx,
\end{equation*}
for any $V \in \mathcal{H}$.
Then 
\begin{equation*}
  \inf_{V\in \mathcal{H}\setminus\{0\}}{ \frac{ B(V)+\frac{N+2-2s}{4
      }\int_{\R_+^{N+1}}y^{1-2s} V^2 G \, dz}
    {\int_{\R_+^{N+1}}y^{1-2s}|\nabla V|^2 G \, dz+\frac{N+2-2s}{4}\int_{\R_+^{N+1}}y^{1-2s} V^2G \, dz}}>0.
\end{equation*}
\end{proposition}
\begin{proof}
  We argue by contradiction assuming that, for any
  $\epsilon \in (0,1)$, there exists $V_\epsilon \in \mathcal{H}$ such
  that
\begin{align}\label{eq-prop-eq-equiv-norm:1}
  &\int_{\R_+^{N+1}}y^{1-2s}|\nabla V_\epsilon|^2 G  \, dz-\mu \int_{\R^N}\frac{|\Tr(V_\epsilon)|^2}{|x|^{2s}} G \, dx 
    +\frac{N+2-2s}{4}\int_{\R_+^{N+1}}y^{1-2s} V_\epsilon^2 G\, dz  \\
  &<\epsilon\left(\int_{\R_+^{N+1}}y^{1-2s}|\nabla V_\epsilon|^2 G \,
    dz
    +\frac{N+2-2s}{4}\int_{\R_+^{N+1}}y^{1-2s} V_\epsilon^2 G\, dz\right),\notag 
\end{align}
i.e.
\begin{equation*}
-\frac{\mu}{1-\epsilon} \int_{\R^N}\frac{|\Tr(V_\epsilon)|^2}{|x|^{2s}} G \, dx 
<-\int_{\R_+^{N+1}}y^{1-2s}|\nabla V_\epsilon|^2 G \, dz-\frac{N+2-2s}{4}\int_{\R_+^{N+1}}y^{1-2s} V_\epsilon^2G \, dz.
\end{equation*} 
Hence by \eqref{ineq-hardy-frac}
\begin{equation*}
\left(\kappa_s\Lambda_{N,s}-\frac{\mu}{1-\epsilon}\right)\int_{\R^N}\frac{|\Tr(V_\epsilon)|^2}{|x|^{2s}} G \, dx < 0. 
\end{equation*}
By \eqref{s-N-def-lambdaNs-kappas-mu}, we conclude that, choosing
$\epsilon<1$ small enough, $\Tr(V_\epsilon)=0$ thus contradicting
\eqref{eq-prop-eq-equiv-norm:1}.
\end{proof}

\begin{proposition}\label{prop-ineq-D+H}
  Let $K>0$. Then there exist a constant
  $C_{N,s,\mu}>0$ depending only on $N$, $s$, $\mu$ and
  $\bar T \in (0,\min\{T,1\})$, depending only on $N$, $s$, $K$, $\mu$, such that,
  for every $\tilde T\in(0,\bar T]$ and any measurable function $f:\R^N\times (0,\tilde T) \to \R$
  satisfying
\begin{equation}\label{hp-h-subhomogeneous-no-nabla}
  |f(x,t)| \le K \left(1+|x|^{-2s+\varepsilon}\right)
  \quad \text{ for a.e. } t \in (0, \tilde T) \text{ and a.e. } x \in \R^N, 
\end{equation}
the following inequality
\begin{multline}\label{ineq-D+H}
  \int_{\R^{N+1}_+} y^{1-2s}|\nabla V|^2 G \, dz-\int_{\R^N}
  \left(\frac{\mu}{|x|^{2s}}|\Tr(V)|^2 +t^sf(\sqrt{t}x,t) |\Tr(V)|^2\right)G(\cdot,0)\, dx \\
  +\frac{N+2-2s}{4}\int_{\R^{N+1}_+} y^{1-2s}V^2 G\, dz \ge
  C_{N,s,\mu}\left(\int_{\R^{N+1}_+} y^{1-2s}|\nabla V|^2 G \,
    dz+\int_{\R^{N+1}_+} y^{1-2s}V^2 G \, dz\right)
\end{multline}
is satisfied for a.e. $t\in (0,\tilde T)$ and for any
$V \in \mathcal{H}$.  Furthermore, there exists a constant
$C'_{N,s}>0$, depending only on $N,s$, such that, for a.e.
$t\in (0,\tilde T)$ and any $V \in \mathcal{H}$,
\begin{multline}\label{ineq-h}
\int_{\R^N} t^s|f(\sqrt{t}x,t)| |\Tr(V)|^2G(\cdot,0)\, dx \\
\le K\, C'_{N,s}  (t^s+t^{\frac{\epsilon}{2}}) \Bigg(\int_{\R^{N+1}_+} y^{1-2s}|\nabla V|^2 G \, dz 
+\int_{\R^{N+1}_+} y^{1-2s}V^2 G \, dz\Bigg).
\end{multline}
\end{proposition}

\begin{proof}
  Thanks to \eqref{hp-h-subhomogeneous-no-nabla}, for any
  $V \in C^\infty_c(\overline{\R^{N+1}_+})$ and a.e. $t\in(0,\tilde
    T)$,
\begin{align*}
  \bigg|\int_{\R^N}&f(\sqrt{t} x,t)|\phi(x,0)|^2G(x,0) \, dx\bigg| \\
  &\le K\int_{\R^N}|\phi(x,0)|^2G(x,0) \, dx+Kt^{-s+\frac{\e}{2}}\int_{\R^N}|x|^{-2s+\varepsilon}|\phi(x,0)|^2G(x,0)
    \, dx\\
  &
    \le  K\int_{\R^N}|\phi(x,0)|^2G(x,0) \, dx +Kt^{-s+\frac{\varepsilon}{2}}\int_{\{|x|\ge
    1\}}|\phi(x,0)|^2G(x,0) \, dx\\&\qquad+Kt^{-s+\frac{\e}{2}}
    \int_{\{|x|\le 1\}}\frac{|\phi(x,0)|^2}{|x|^{2s}}G(x,0) \, dx \notag \\
  &\le
    \frac{K}{t^s}\left(t^s+t^{\frac{\varepsilon}{2}}\right)\int_{\R^N}|\phi(x,0)|^2G(x,0)
    \, dx
    +Kt^{-s+\frac{\e}{2}}\int_{\R^N}\frac{|\phi(x,0)|^2}{|x|^{2s}}G(x,0) \, dx. \notag
\end{align*}
Then, in view of \eqref{ineq-trace}, \eqref{ineq-hardy-frac}, a
density argument implies \eqref{ineq-h}. From 
Proposition \ref{prop-equiv-norm} and \eqref{ineq-h}, choosing
$\bar T$ small enough, we deduce \eqref{ineq-D+H}.
\end{proof}

\begin{proposition}
There exists a constant
  $C''_{N,s}>0$, depending only on $N$ and $s$, such that, for any
 $\rho \in L^{\frac{N}{2s}}(\R^N)$ and 
  $V \in \mathcal{H}$,
\begin{equation}\label{ineq-h_t}
  \int_{\R^N} |\rho| |\Tr(V)|^2 G(\cdot,0) \, dx
  \le C''_{N,s}  \norm{\rho}_{L^{\frac{N}{2s}}(\R^N)}\norm{V}_{\mathcal{H}}^2.
\end{equation}
\end{proposition}
\begin{proof}
  By Proposition \ref{prop-H-t-embbeded-H1} and \eqref{trace-Hs},
  $\Tr(V) \sqrt{G(\cdot,0)} \in W^{s,2}(\R^N)$. Hence, thanks to the
  Fractional Sobolev Embedding Theorem,
  $\Tr(V) \sqrt{G(\cdot,0)} \in L^{\frac{2N}{N-2s}}(\R^N)$ and there
  exists a constant $S_{N,s}>0$, which depends only on $N$ and $s$, such
  that
 \begin{equation*}
   \norm{\Tr(V) \sqrt{G(\cdot,0)}}_{L^{\frac{2N}{N-2s}}(\R^N)}
   \le S_{N,s} \norm{\Tr(V)\sqrt{G(\cdot,0)}}_{W^{s,2}(\R^N)}.
 \end{equation*} 
Furthermore, by the Hold\"er inequality, 
\begin{equation*}
  \int_{\R^N}  |\rho| |\Tr(V)|^2 G(\cdot,0) \, dx
  \le  \norm{\rho}_{L^{\frac{N}{2s}}(\R^N)}
  \norm{\Tr(V)\sqrt{G(\cdot,0)}}_{L^{\frac{2N}{N-2s}}(\R^N)}^2
\end{equation*}
and so  \eqref{ineq-h_t} follows from   \eqref{trace-Hs} and  \eqref{ineq-VsqrtG}.
\end{proof}

\section{An alternative formulation in Gaussian spaces} \label{sec-formulation}

In this section we present an alternative formulation of
\eqref{eq-weak-frac-extended-forward} and a regularity
result. Henceforth, for the sake of simplicity, we will assume that
$t_0=0$; this is not restrictive up to a translation. We deal with the
backward version of \eqref{prob-frac-extended-forward} which is
completely equivalent to \eqref{prob-frac-extended-forward}.  Let
\begin{align} 
  &h(x,t):=g(x,-t) \quad\text{for any } t \in (0,T)
    \text{ and a.e. } x \in \R^{N},\label{def-h}\\
  &U(z,t):=W(z,-t) \text{ for a.e.}  t \in (0,T)
    \text{ and  } z \in \R^{N+1}_+ \label{def-U},\\
  &u:=\Tr(U) \label{def-u}.
\end{align}
Then $U \in L^2(\R,H^1(\R^{N+1}_+,y^{1-2s}))$ and $U$ weakly solves
\begin{equation}\label{prob-frac-extended}
\begin{cases}
  y^{1-2s}U_t+\dive(y^{1-2s}\nabla U)=0,
  &\text{ in } \R_+^{N+1} \times (0,T),\\
  \Tr( U(\cdot,t))= u(\cdot,t),
  &\text{ on }\R^{N}, \text{ for a.e. } t \in  (0,T),\\
  -\lim\limits_{y \to 0^+}y^{1-2s}\pd{ U}{y}= \frac{\mu}{|x|^{2s}} u
  +\ h u &\text{ on }\R^N\times (0,T),
\end{cases}
\end{equation}
in the sense that 
\begin{multline}\label{eq-weak-frac-extended}
  \int_{0}^{T}\left(\int_{\R_+^{N+1}}y^{1-2s} U\phi_t  \, dz \, \right)dt\\
  =-\int_{0}^{T}\left(\int_{\R_+^{N+1}}y^{1-2s} \nabla U \cdot \nabla
    \phi \, dz \,
  \right)dt+\int_0^T\left(\int_{\R^N}\left(\frac{\mu}{|x|^{2s}} u
      \phi+h u\phi \right)dx \right) dt,
\end{multline}
for any $\phi \in C^\infty_c (\overline{\R_+^{N+1}} \times(0,T))$, if
and only if $W$ is a weak solution to problem
\eqref{prob-frac-extended-forward}.

\begin{definition}\label{def-weak-derivative}
  Let $X$ be a Hilbert space and $[T_1,T_2] \subset \R$.  A function
  $U \in L^2((T_1,T_2),X)$ has a weak derivative
  $\Psi\in L^2((T_1,T_2), X)$ if, for any
  $\phi \in C^\infty_c((T_1,T_2),X)$,
\begin{equation*}
	\int_{T_1}^{T_2}\ps{X}{U}{\phi_t}dt=-\int_{T_1}^{T_2}\ps{X}{\Psi}{\phi}dt.
\end{equation*}
Furthermore we define
\begin{equation*}
  H^1((T_1,T_2), X):=\{U \in L^2((T_1,T_2),X):
  U \text{ has  a weak derivative } \Psi\in L^2((T_1,T_2), X)\}.
\end{equation*}
\end{definition}

Let $(X,L,X^*)$ be a Hilbert triplet. Thanks to the Riesz isomorphism,
the property that $U \in L^2((T_1,T_2),X^*)$ has a weak derivative
$\Psi \in L^2((T_1,T_2),X^*)$ can be rephrased equivalently as
\begin{equation*}
  \int_{T_1}^{T_2}\df{X^*}{U(t)}{\phi_t(t)}{X}dt=
  -\int_{T_1}^{T_2}\df{X^*}{\Psi(t)}{\phi(t)}{X}dt
\end{equation*}
for any $\phi \in C^\infty_c((T_1,T_2),X)$. Then the property that a
function $U \in L^2((T_1,T_2),X)$ has a weak derivative
$\Psi \in L^2((T_1,T_2),X^*)$ is equivalent to the fact that
\begin{equation}\label{def-weak-derivative-respect-to-t-H1*}
  \int_{T_1}^{T_2}\ps{L}{U(t)}{\phi_t(t)}dt
  =-\int_{T_1}^{T_2}\df{X^*}{\Psi(t)}{\phi(t)}{X}dt
\end{equation}
for any $\phi \in C^\infty_c((T_1,T_2),X)$.

For any $U\in L^2(\R,H^1(\R^{N+1}_+,y^{1-2s}))$ satisfying
\eqref{eq-weak-frac-extended}, let us consider the function
\begin{equation}\label{def-V}
V(z,t):=U(\sqrt{t}z,t).
\end{equation} 
By a density argument, we can easily verify that 
\begin{equation*}
v(\cdot,t):=\Tr(U(\sqrt t\cdot ,t))=u(\sqrt{t}\cdot,t).
\end{equation*} 
In Proposition \ref{prop-eq-Vt} below, we derive the weak
  formulation of  the problem
  solved by $V$.
\begin{proposition}\label{prop-eq-Vt}
  Let $U\in L^2(\R,H^1(\R^{N+1}_+,y^{1-2s}))$ be a solution of
  \eqref{eq-weak-frac-extended}. Then, letting $V$ be as in
  \eqref{def-V},
\begin{equation}\label{V-H*}
  V \in L^2((\tau,T),\mathcal{H}), \quad  V_t \in
  L^2((\tau,T),\mathcal{H}^*)
  \quad \text{ for any } \tau \in (0,T),
\end{equation} 
 and 
\begin{multline}\label{eq-weak-formulation-gaussian}
  \df{\mathcal{H^*}}{V_t}{\phi}{ \mathcal{H}}=\frac{1}{t}
  \int_{\R^{N+1}_+}y^{1-2s}\nabla V \cdot \nabla \phi \, G \, dz \\
  -\frac{1}{t}\int_{\R^N}\left(\frac{\mu}{|x|^{2s}}
    v(x,t)\phi(x,0)+t^sh(\sqrt{t}x,t)v(x,t)\phi(x,0)\right) G(x,0)\,
  dx,
\end{multline} 
for any $\phi \in C^\infty_c(\overline{\R_+^{N+1}})$ and for a.e. $t \in (0,T)$.
\end{proposition}

\begin{proof}
  Let $\phi \in C^\infty_c (\overline{\R_+^{N+1}}
  \times(0,T))$. Testing \eqref{eq-weak-frac-extended} with $\phi G_s$
  we obtain
\begin{multline*}
  \int_{0}^{T}\left(\int_{\R_+^{N+1}}y^{1-2s} U\phi
    \left[-\frac{N+2-2s}{2t} +\frac{|z|^2}{4t^2}\right]G_s \, dz \,
  \right)dt+\int_{0}^{T}
  \left(\int_{\R_+^{N+1}}y^{1-2s} U\phi_t G_s \, dz  \right) dt\\
  =\int_{0}^{T}\left(\int_{\R_+^{N+1}}y^{1-2s} \nabla U \cdot
    \frac{z}{2t} \phi \, G_s dz\right)
  dt-\int_{0}^{T}\left(\int_{\R_+^{N+1}}y^{1-2s}
    \nabla  U \cdot \nabla \phi \, G_s dz \, \right)dt\\
  +\int_0^T\left(\int_{\R^N}\left(\frac{\mu}{|x|^{2s}} u(x) \phi(x,0,t)
      +h(x,t) u(x)\phi(x,0,t) \right)G_s(x,0,t)\, dx \right) dt.
\end{multline*}
Let $\tilde{\phi}(z,t):= \phi(\sqrt{t}z,t)$. Then the   change of variables $z=\sqrt{t}z'$ yields
\begin{multline}\label{prop-eq-Vt:2}
  \int_{0}^{T}\left(\int_{\R_+^{N+1}}y^{1-2s} V\tilde{\phi}
    \left[-\frac{N+2-2s}{2t}+\frac{|z|^2}{4t}\right]G  \, dz \, \right)dt\\
  +\int_{0}^{T}\left( \int_{\R_+^{N+1}}y^{1-2s} V\left[\tilde{\phi}_t
      -\nabla \tilde{\phi}\cdot\frac{z}{2t}\right]G    \, dz \, \right)dt\\
  =\int_{0}^{T}\left(\int_{\R_+^{N+1}}y^{1-2s} \nabla V \cdot
    \frac{z}{2t} \tilde{\phi} \, G dz\right) dt-
  \int_{0}^{T}\frac{1}{t}\left(\int_{\R_+^{N+1}}y^{1-2s}
    \nabla  V \cdot \nabla \tilde{\phi} \, G   dz \, \right)dt\\
  +\int_0^T\frac{1}{t}\left(\int_{\R^N}\left(\frac{\mu}{|x|^{2s}} v
      \tilde{\phi}(x,0,t) +t^sh(\sqrt{t}x,t) v\tilde{\phi}(x,0,t)
    \right)G \, dx \right) dt
\end{multline}
for any
$\tilde\phi \in C^\infty_c (\overline{\R_+^{N+1}} \times(0,T))$,
where $V$ is defined in \eqref{def-V}.

Let $\R_\delta^{N+1}$ be as in
\eqref{R-N-delta} for any $\delta >0$.  Then, by the Dominated
Convergence Theorem,
\begin{equation*}
  \int_{\R_+^{N+1}}y^{1-2s} V (\nabla \tilde{\phi}\cdot z) G  \, dz
  =\lim_{\delta \to 0^+}\int_{\R_{\delta}^{N+1}}y^{1-2s}
  V (\nabla \tilde{\phi}\cdot z)G    \, dz.
\end{equation*}
Furthermore, since 
\begin{equation*}
  \dive\left(y^{1-2s}V\tilde{\phi} G z\right)=y^{1-2s}\left[(N+2-2s)V
    \tilde{\phi}
    +(\nabla V \cdot z) \tilde\phi +V(\nabla \tilde{\phi}\cdot z)
    -V \tilde{\phi}\, \frac{|z|^2}{2}\right]G , 
\end{equation*}
an integration by parts on $\R^{N+1}_\delta$ yields 
\begin{multline} \label{prop-eq-Vt:3.1}
  \int_{\R_{\delta}^{N+1}}y^{1-2s} V (\nabla \tilde{\phi}\cdot z) G
  \, dz
  =-(N+2-2s)\int_{\R_{\delta}^{N+1}} y^{1-2s}V \tilde{\phi} G  \, dz\\
  -\int_{\R_{\delta}^{N+1}}y^{1-2s}(\nabla V \cdot z) \tilde{\phi} G\,
  dz+\int_{\R_{\delta}^{N+1}}y^{1-2s} V \tilde{\phi} \,\frac{|z|^2}{2}G
  \,dz
  -\delta^{2-2s}\int_{\R^N}V(x,\delta,t)\tilde{\phi}(x,\delta,t)G(x,\delta)
  \, dx.
\end{multline}
We claim that 
\begin{equation}\label{prop-eq-Vt:4}
  \liminf_{\delta \to 0^+}\delta^{2-2s}
  \int_{\R^N}V(x,\delta,t)\tilde{\phi}(x,\delta,t)G(x,\delta) \, dx=0.
\end{equation}
To prove \eqref{prop-eq-Vt:4} we argue by contradiction.  If
\eqref{prop-eq-Vt:4} does not hold, then there exists a constant $C>0$
and $ \bar \delta \in (0,+\infty)$ such that
\begin{equation*}
  \delta^{1-2s}\int_{\R^N}V(x,\delta,t)\tilde{\phi}(x,\delta,t)G(x\delta)
  \, dx>\frac{C}{\delta}
\end{equation*}
for any $\delta \in (0,\bar \delta)$. Integrating on
$(0,\bar \delta)$, we obtain, thanks to the Fubini-Tonelli Theorem,
\begin{multline*}
+\infty >\int_{\R_+^{N+1}}y^{1-2s} V\tilde{\phi}G  \, dz
\ge \int_0^{\bar \delta}\left(\int_{\R^N}y^{1-2s} V \tilde{\phi} G  \, dx\right)\, dy\ge
C\int_0^{\bar \delta} \frac{1}{y} \, dy=+\infty,
\end{multline*}
which is a contradiction.  Hence there exists a sequence
$\delta_n \to 0 ^+$ such that, passing to the limit as
$ \delta =\delta_n$ and $n \to \infty$ in \eqref{prop-eq-Vt:3.1}, we
obtain that, for a.e. $t\in(0,T)$, 
\begin{multline}\label{prop-eq-Vt:5}
  \int_{\R_+^{N+1}}y^{1-2s} V (\nabla \tilde{\phi}\cdot z) G  \, dz
  =-(N+2-2s)\int_{\R_+^{N+1}}y^{1-2s} V \tilde{\phi} G\, dz\\
  -\int_{\R_+^{N+1}}y^{1-2s}(\nabla V \cdot z) \tilde{\phi} G\,
  dz+\int_{\R_+^{N+1}}y^{1-2s} V \tilde{\phi} \frac{|z|^2}{2}G \, dz.
\end{multline}
Putting together \eqref{prop-eq-Vt:2} and  \eqref{prop-eq-Vt:5}, we
conclude that 
\begin{multline*}
  \int_{0}^{T}\left(\int_{\R_+^{N+1}}y^{1-2s} V \tilde{\phi}_t \, dz
  \right) dt
  =-\int_{0}^{T}\left(\frac{1}{t}\int_{\R^{N+1}_+}y^{1-2s}\nabla V
    \cdot \nabla \tilde{\phi}\,G \, dz\right) dt \\
  +\int_{0}^{T}\left(\frac{1}{t} \int_{\R^N}\left(\frac{\mu}{|x|^{2s}}
      v(x,t)\tilde{\phi}(x,0,t)
      +t^sh(\sqrt{t}x,t)v(x,t)\tilde{\phi}(x,0,t)\right)G(x,0)\,
    dx\right) dt.
\end{multline*} 
The integrand at the right hand side of the above equation belongs to
$\mathcal{H}^*$ as a function of $\tilde{\phi}$ for a.e.
$t \in (\tau,T)$ in view of \eqref{hp-g-subhomogeneous}, \eqref{ineq-trace},
\eqref{ineq-hardy-frac}, \eqref{def-h} 
 and the H\"older inequality. Hence, in view of
\eqref{def-weak-derivative-respect-to-t-H1*}, 
we conclude that \eqref{V-H*} and \eqref{eq-weak-formulation-gaussian}
are satisfied.
\end{proof}

\begin{remark}\label{remark-parabolic-change-variables}
  From the theory of abstract parabolic equations, see for example
  \cite[Theorem 8.60]{LG} and \cite[Theorem 1, p. 473, Theorem 2,
  p. 477]{DL}, if $V$ satisfies \eqref{V-H*}, then
\begin{align*}
  &V \in C^0([\tau,T], \mathcal{L}), \text{ for any } \tau \in (0,T),\\
  &t \to \norm{V(\cdot,t)}^2_{\mathcal{L}}
    \text{ is absolutely continuous on } [\tau,T] \text{ for any } \tau \in (0,T),\\
  &\sideset{_{\mathcal{H}^*}}{_{\mathcal{H}}}{\mathop{\left\langle
    V_t(\cdot,t),
    V(\cdot,t)\right\rangle}}=
    \frac{1}{2}\frac{d}{dt}\norm{V(\cdot,t)}^2_{\mathcal{L}}=\frac{1}{2}\frac{d}{dt}\int_{\R^{N+1}_+} y^{1-2s}V^2G \, dz 
\end{align*}
in a distributional sense and  for a.e. $t \in (0,T)$.
 More in general, if $V,W$ satisfies \eqref{V-H*}, then 
\begin{align*}
  &t \to \ps{\mathcal{L}}{V(\cdot,t)}{W(\cdot,t)}
    \text{ is absolutely continuous on } [\tau,T]
    \text{ for any } \tau \in (0,T), \notag\\
  &\sideset{_{\mathcal{H}^*}}{_{\mathcal{H}}}{\mathop{\left\langle
    V_t(\cdot,t),
    W(\cdot,t)\right\rangle}}
    +\sideset{_{\mathcal{H}^*}}{_{\mathcal{H}}}{\mathop{\left\langle
    W_t(\cdot,t),
    V(\cdot,t)\right\rangle}}=\frac{d}{dt}\ps{\mathcal{L}}{V(\cdot,t)}{W(\cdot,t)} 
\end{align*}
in a distributional sense and  for a.e. $t \in (0,T)$.
\end{remark}

\begin{proposition} \label{prop-trace-H1(0,T)} Let
  $(\mathop{\rm{\widetilde{T}r}}U)(\cdot,t):=\Tr(U(\cdot,t))$ for any
  $U \in H^1((0,T),\mathcal{H})$. Then
\begin{equation*}
  \mathop{\rm{\widetilde{T}r}}: H^1((0,T), \mathcal{H}) \to H^1((0,T),L^2(\R^N,|x|^{-2s}G)) 
	\end{equation*}
	is a linear and continuous trace operator such that  
	\begin{equation*}
		(\mathop{\rm{\widetilde{T}r}}(U))_t(\cdot,t)=\mathop{\rm{Tr}}(U_t(\cdot,t)), \text{ for any  } U \in H^1((0,T), \mathcal{H}) \text{ and } \text{ a.e. } t \in (0,T).
	\end{equation*}
\end{proposition}
\begin{proof}
  In view of \eqref{ineq-hardy-frac} we have that
  $\mathop{\rm{\widetilde{T}r}}(U) \in
  L^2((0,T),L^2(\R^N,|x|^{-2s}G(x,0)))$. Furthermore, there exists a
  sequence
  $\{U_n\}_{n \in \mathbb{N}} \subset C^\infty([0,T],\mathcal{H})$
  such that $U_n \to U$ as $n \to \infty$ in $H^1((0,T), \mathcal{H})$
  thanks to \cite[Lemma 2.5.6.]{HNVML}.

	Let us prove that $\mathop{\rm{\widetilde{T}r}}(U_n) \in C^1([0,T], L^2(\R^N,|x|^{-2s}G(x,0)))$ and that 
	\begin{equation*}
		(\mathop{\rm{\widetilde{T}r}}(U_n))_t(\cdot,t)=\mathop{\rm{Tr}}((U_n)_t(\cdot,t)) \quad \text{ for any } t \in (0,T).
	\end{equation*} 
	We start by showing that the incremental ratio of $\mathop{\rm{\widetilde{T}r}}(U_n)(\cdot,t)$ tends to  $\Tr((U_n)_t(\cdot,t))$ strongly in  $L^2(\R^N,|x|^{-2s}G(x,0))$ for any $t \in (0,T)$.
	Let $ t\in (0,T)$ and $h \in \R$ be such that $|h| \le \min\{t,T-t\}$.
	Then, by definition of $\mathop{\rm{\widetilde{T}r}}$ and linearity,
	\begin{multline*}
          \int_{\R^N}\frac{G(x,0)}{|x|^{2s}}
          \left|\left(
              \frac{\mathop{\rm{\widetilde{T}r}}(U_n)(\cdot,t+h)
                -\mathop{\rm{\widetilde{T}r}}(U_n)(\cdot,t)}{h}
            \right)
            -\mathop{\rm{\widetilde{T}r}}((U_n)_t (\cdot,t))\right|^2\,dx \\
          =\int_{\R^N}\frac{G(x,0)}{|x|^{2s}} \left|
            \Tr\left(\frac{U_n(\cdot,t+h)-U_n(\cdot,t)}{h} -(U_n)_t(\cdot,t)\right)\right|^2\,dx\\
          \le {\rm{const}} \norm{\frac{U_n(\cdot,t+h)-U_n(\cdot,t)}{h}
            -(U_n)_t (\cdot,t)}_{\mathcal{H}}^2 \to 0
	\end{multline*}
	as $h \to 0^+$, in view of \eqref{ineq-hardy-frac}.
	
	In the same way we can show that
        $(\mathop{\rm{\widetilde{T}r}}(U_n))_t \in
        C^0([0,T],L^2(\R^N,|x|^{-2s}G(x,0)))$.  It follows that,
        taking the limit of the incremental ratio,
	\begin{align}
          \frac{d}{dt}&\ps{L^2(\R^N,|x|^{-2s}G(x,0))}{{\rm{\widetilde{T}r}}(U_n)}{\phi}\\
                      &=\ps{L^2(\R^N,|x|^{-2s}G(x,0))}{{\rm{Tr}}((U_n)_t)}{\phi}
                        +\ps{L^2(\R^N,|x|^{-2s}G(x,0))}{{\rm{\widetilde{T}r}}(U_n)}{\phi_t} 
        \end{align}
	for any  function $\phi \in C_c^\infty((0,T),L^2(\R^N,|x|^{-2s}G(x,0)))$.

	Then, for any test function $\phi \in C_c^\infty((0,T),L^2(\R^N,|x|^{-2s})G(x,0)))$,
	\begin{multline*}
		\int_0^T\left(\int_{\R^N}\frac{\mathop{\rm{\widetilde{T}r}}(U)}{|x|^{2s}}\phi_tG(x,0) \, dx \right)\, dt=\lim_{n \to \infty}\int_0^T\left(\int_{\R^N}\frac{\mathop{\rm{\widetilde{T}r}}(U_n)}{|x|^{2s}}\phi_tG(x,0) \, dx \right)\, dt\\
		=-\lim_{n \to \infty}\int_0^T\left(\int_{\R^N}\frac{\mathop{\rm{Tr}}((U_n)_t)}{|x|^{2s}}\phi G(x,0) \, dx \right)\, dt=\int_0^T\left(\int_{\R^N}\frac{\mathop{\rm{Tr}}(U_t)}{|x|^{2s}}\phi G(x,0)\, dx \right)\, dt.
	\end{multline*}
	We  conclude that there exists the weak derivative with respect to $t$ of $\mathop{\rm{\widetilde{T}r}}(U)$ and that 
	\begin{equation*}
		(\mathop{\rm{\widetilde{T}r}}(U))_t(\cdot,t)=\Tr(U_t(\cdot,t))  \text{ for  a.e. } t \in (0,T).
	\end{equation*}
	The continuity of the operator follows from  \eqref{ineq-hardy-frac}.
\end{proof}

\begin{remark}\label{remark-Lt-Ht*}
The  natural embedding
\begin{equation*}
  I:\mathcal{L} \to \mathcal{H}^*, \quad I(V)(W):=\int_{\R^{N+1}_+} y^{1-2s} V W G\, dz
\end{equation*}
is linear, continuous and injective. With a slight abuse of notation,
we will identify $\mathcal{L}$ and $I(\mathcal{L})$.
\end{remark}

\begin{proposition}\label{prop-Vt-L2}
For $K>0$, let $\bar T \in (0,\min\{T,1\})$ depending on $K$  be as in
Proposition \ref{prop-ineq-D+H}. Let  $\tilde T\in(0,\bar T]$ and
$f\in L^1_{loc}(\R^N\times(0,\tilde T))$ be such that
\begin{align}
&|f(x,t)|+|\nabla f(x,t)\cdot x|\le  K(1+|x|^{-2s+\varepsilon})
                \text{ for a.e. } t \in (0,\tilde T) \text{ and a.e. } x \in \R^N,\label{hp-f-subhomogeneous}\\
& f_t\in L^\infty_{loc}((0,\tilde T), L^{\frac{N}{2s}}(\R^N)). \notag  
\end{align}
If $ \tau \in (0,\tilde T)$, $V \in L^2((\tau,\tilde T),\mathcal{H})$,
$V_t \in L^2((\tau,\tilde T),\mathcal{H}^*)$,
$V(\cdot,\tilde T)\in\mathcal{H}$, and $V$ is a solution of
\eqref{eq-weak-formulation-gaussian} with $h=f$, then
$V_t \in L^2((\tau,\tilde T), \mathcal{L})$ in the sense of Remark
\ref{remark-Lt-Ht*}.
\end{proposition}
\begin{proof}
For all $t\in(0,\tilde T)$, let us consider the linear map
\begin{align*}
  &A_t:\mathcal{H}\to  \mathcal{H}^*,  \\
  &\df{\mathcal{H}*}{A_t(V)}{\phi}{\mathcal{H}}:
    =\int_{\R^{N+1}_+}y^{1-2s} \nabla V\cdot \nabla \phi \,G\, dz \\
  &-\int_{\R^N}\left(\frac{\mu}{|x|^{2s}}\Tr(V) \Tr(\phi)
    +t^sf(\sqrt{t}x,t)\Tr(V) \Tr(\phi) \right) G\,dx
    \quad \text{ for any } \phi, V \in \mathcal{H}. 
\end{align*}
In view of \eqref{hp-f-subhomogeneous}, the H\"older inequality,
\eqref{ineq-trace} and \eqref{ineq-hardy-frac}, $A_t$ is well
defined and continuous.  From standard techniques in the theory
  of parabolic equations, see for example \cite{E}, the
Faedo-Galerkin method provides a sequence of functions
$\{V_n\}_{n \in \mathbb{N}}$ such that:
\begin{align}
  &V_n \in L^2((\tau,\tilde T), \mathcal{H}) \text{ for any } n \in \mathbb{N}, \notag  \\
  &V_n \rightharpoonup V \text{ weakly  in  } L^2((\tau,\tilde T),\mathcal{H}) \text{ as } n \to \infty, \label{eq:1} \\
  &(V_n)_t \in L^2((\tau,\tilde T), \mathcal{H})  \text{ for any } n \in \mathbb{N}, \notag   \\
  &(V_n)_t \rightharpoonup V_t \text{ weakly in  }  L^2((\tau,\tilde T), \mathcal{H}^*) \text{ as } n \to \infty,\label{prop-Vt-L2:7} \\
  &V_n(\cdot,\tilde T) \to V(\cdot, \tilde T) \text{ strongly   in  } \mathcal{H} \text{ as } n \to \infty,  \\
  &\{V_n\}_{n \in \mathbb{N}} \text{ is bounded in } C([\tau, \tilde T],\mathcal{L}).  \label{prop-Vt-L2:9}
\end{align}

For any $n \in \mathbb{N}$, the function $V_n$
belongs to $H^1((\tau,\tilde T),W_n)$ and solves, for a.e. $ t \in
  (0,\tilde T)$,
\begin{multline}\label{prop-Vt-L2:1.1}
  t\int_{\R^{N+1}_+}y^{1-2s}(V_n)_t \phi \,G \,
  dz=\int_{\R^{N+1}_+}y^{1-2s}\nabla V_n \cdot \nabla \phi \,G \, dz
  \\-\int_{\R^N}\left(\frac{\mu}{|x|^{2s}}
    \Tr(V_n)\Tr(\phi)+t^sf(\sqrt{t}x,t)\Tr(V_n)\Tr(\phi)\right)G(x,0)\,
  dx,
\end{multline}
for any $\phi \in W_n\subset \mathcal{H}$, where $W_n$ is a suitable
finite dimensional subspace of $\mathcal{H}$.  Testing
\eqref{prop-Vt-L2:1.1} with $(V_n)_t$ and integrating with
respects to $t$ on $(\tau,\tilde T)$, we obtain that 
\begin{multline*}
  \int_{\tau}^{\tilde T}t\left(\int_{\R^{N+1}_+}y^{1-2s}|(V_n)_t|^2
    \,G \, dz\right)\, dt
  =\int_{\tau}^{\tilde T}\df{\mathcal{H}*}{A_t(V_n)}{(V_n)_t}{\mathcal{H}} \, dt \\
  =\frac{1}{2}\df{\mathcal{H}*}{A_{\tilde T}(V_n)(\tilde T )}{V_n(\tilde
    T)}{\mathcal{H}}
  -\frac{1}{2}\df{\mathcal{H}*}{A_{\tau}(V_n)(\tau)}{V_n(\tau)}{\mathcal{H}} \\
  +\frac12\int_\tau^{\tilde
    T}\left(\int_{\R^N}\left[st^{s-1}f(\sqrt{t}x,t)+\frac{1}{2}t^{s-\frac{1}{2}}
      \nabla f(\sqrt{t}x,t) \cdot x +
      t^sf_t(\sqrt{t}x,t)\right]|\Tr(V_n)|^2 \, G(\cdot,0)dx\right) \,
  dt
\end{multline*}
thanks to Proposition \ref{prop-trace-H1(0,T)}. For a.e. $t \in (\tau,\tilde T)$, 
\begin{equation}\label{prop-Vt-L2:3}
  \int_{\R^N}\left|st^{s-1}f(\sqrt{t}x,t)+\frac{1}{2}t^{s-\frac{1}{2}}
    \nabla f(\sqrt{t}x,t) \cdot x\right| |\Tr(V_n)|^2 \, G(\cdot,0)dx \le {\rm{const}} \norm{V_n}^2_{\mathcal{H}}
\end{equation}
 in view of  \eqref{ineq-trace},   \eqref{ineq-hardy-frac}, and \eqref{hp-f-subhomogeneous}.

 By \eqref{ineq-D+H},  \eqref{ineq-h_t}, \eqref{eq:1}, \eqref{prop-Vt-L2:9}, and
 \eqref{prop-Vt-L2:3}, and we conclude that
 $\{(V_n)_t\}_{n \in \mathbb{N}}$ is bounded in
 $L^2((\tau, \tilde T),\mathcal{L})$.  Then, up to a subsequence, there
 exists $W \in L^2((\tau,\tilde T),\mathcal{L})$ such that
\begin{equation*}
	(V_n)_t \rightharpoonup W\text{ weakly in } L^2((\tau,\tilde T),\mathcal{L}).
\end{equation*}
By \eqref{prop-Vt-L2:7} we conclude that $W=V_t$, hence  $V_t \in  L^2((\tau,\tilde T),\mathcal{L})$.
\end{proof}

\section{Spectrum of a weighted Ornstein-Uhlenbeck
  operator} \label{sec-eigenvalues} In this section we prove
Proposition \ref{prop-eigenvalues}. The following compactness result ensures that the point spectrum
of \eqref{prob-eigenvalue-Ornstein-Uhlenbeck-operator} is discrete.
\begin{proposition}\label{prop-embedding}
The embedding $i:\mathcal{H} \to \mathcal{L}$, $i(V)=V$,
is compact.
\end{proposition}

\begin{proof}
  Let $\{V_n\}$ be a sequence converging to some $V \in \mathcal{H}$
  weakly in $\mathcal{H}$ as $n \to \infty$.  Then by \cite[Theorem
  19.7]{Opic-Kufner}, for any $R>0$
\begin{equation}\label{prop-embedding:1}
\lim_{n \to \infty}\int_{B_R^+}y^{1-2s}|V-V_n|^2  G\, dz=0.
\end{equation}
Moreover, for any $n \in \mathbb{N}$ and   $R>0$, by \eqref{ineq-hardy-extended},
\begin{multline}\label{prop-embedding:2}
  \int_{\R^{N+1}_+\setminus B_R^+}y^{1-2s}|V-V_n|^2  G dz
  \le \frac{1}{R^2}\int_{\R^{N+1}_+\setminus B_R^+}y^{1-2s}|z|^2|V-V_n|^2  G dz\\
\le  \frac{16}{ R^2}\int_{\R_+^{N+1}}y^{1-2s} |\nabla (V-V_n)|^2G \,
dz
+ \frac{4(N+2-2s)}{R^2}\int_{\R_+^{N+1}}y^{1-2s}  |V-V_n|^2G \, dz. 
\end{multline}
Since  $\{V_n\}_{n \in \mathbb{N}}$ is bounded in $\mathcal{H}$ we conclude by \eqref{prop-embedding:2} that
\begin{equation}\label{prop-embedding:3}
\int_{\R^{N+1}_+\setminus B_R^+}y^{1-2s}|V-V_n|^2  G dz \le \frac {\co}{R^2} \quad \text{ for any } R>0 \text{ and for any } n \in \mathbb{N}.
\end{equation}
Putting together \eqref{prop-embedding:1} and \eqref{prop-embedding:3} we obtain that $V_n \to V$ strongly in $\mathcal{L}$ as $n \to +\infty$, thus completing the proof.
\end{proof}

\begin{proposition}\label{prop-basis-eigenfunctions}
  The eigenvalues of
  \eqref{prob-eigenvalue-Ornstein-Uhlenbeck-operator} form a
  non-decreasing, diverging sequence
  $\{\gamma_k\}_{k \in \mathbb{N}\setminus \{0\}}$. Furthermore there
  exists an orthonormal basis of $\mathcal{L}$ of eigenfunctions of
  \eqref{prob-eigenvalue-Ornstein-Uhlenbeck-operator} whose elements
  belong to $\mathcal{H}$.
\end{proposition}

\begin{proof}
Let $L:\mathcal{H} \to \mathcal{H}^*$ be defined as 
\begin{multline*}
L(V)(\phi):=\int_{\R^{N+1}_+}y^{1-2s} \nabla U \cdot\nabla \phi \, G\, dz\\
-\int_{\R^N} \frac{\mu}{|x|^{2s}}\Tr(U)\Tr(\phi)   \, G\, dx+
\frac{N+2-2s}{4}\int_{\R^{N+1}_+}y^{1-2s}   V \phi \, G\, dz\\
\end{multline*}
for any $V,\phi \in \mathcal{H}$.  In view of \eqref{ineq-D+H} in the
case $f\equiv 0$, the operator $L$ is coercive. It follows that the
operator $T:\mathcal{L}\to \mathcal{L}$ defined as 
$T:=L^{-1}$ is
well-defined. Since $T$ is also compact in view of Proposition
\ref{prop-embedding}, the conclusion follows from by the Spectral Theorem.
\end{proof}

\begin{remark} \label{remark-def-trace-Sr+}
For any $r>0$, there exists a linear, continuous and compact trace operator 
\begin{equation*}
\mathop{{\rm{Tr}}_{S_r^+}}:\mathcal{H} \to L^2(S_r^+,y^{1-2s}).
\end{equation*}
Indeed $\mathcal{H} \hookrightarrow H^1(B_r^+,y^{1-2s})$ since
$G>{ \rm{const}} >0$ on $B_r^+$; moreover, in view of \cite[Theorem
19.7]{Opic-Kufner} and the Divergence Theorem, one can easily verify
that the trace operator from
$H^1(B_r^+,y^{1-2s})$ to $L^2(S_r^+, y^{1-2s})$ is compact.
\end{remark}

\begin{proof}[\textbf{Proof of Proposition \ref{prop-eigenvalues}}]
  Let $\gamma$ be an eigenvalue of problem
  \eqref{prob-eigenvalue-Ornstein-Uhlenbeck-operator} and $Y$ an
  associated eigenfunction.  Let
  $\{\psi_k\}_{k\in \mathbb{N} \setminus \{0\}}$ be the orthonormal
  basis of $L^2(\mathbb{S}^N_+,\theta_{N+1}^{1-2s})$ introduced in
  Section
  \ref{sec-functional-settings}. By Remark \ref{remark-def-trace-Sr+},
  for any $r >0$ the function $Y$ admits a trace
  $\mathop{{\rm{Tr}}_{S_r^+}}(Y) \in
  L^2(S_r^+,  y^{1-2s})$. By a change of variables
  $\mathop{{\rm{Tr}}_{\mathbb{S^N}_+}}(Y(r\cdot))\in
  L^2(\mathbb{S}^N_+,\theta_{N+1}^{1-2s})$, hence 
\begin{equation*}
  \mathop{{\rm{Tr}}_{\mathbb{S^N}_+}}(Y(r\cdot))
  =\sum_{k=1}^\infty \varphi_k(r) \psi_k, \quad 
 \text{ with } \varphi_k(r):=
 \int_{\mathbb{S}^N_+}\theta_{N+1}^{1-2s}\,
\mathop{{\rm{Tr}}_{\mathbb{S^N}_+}}(Y(r\theta))
  \,\psi_k (\theta)\, dS.
\end{equation*} 
Since, by classical elliptical regularity theory,
$Y \in C^{\infty}_{loc}(\R^{N+1}_+)$, so writing $z=r \theta$ where
$r=|z|$ and $\theta=\frac{z}{|z|}$, we have that
$Y(z)=Y(r
\theta)=\mathop{\rm{Tr}_{\mathbb{S^N}_+}}(Y(r\cdot))|_{\theta}
=\sum_{k=1}^\infty \varphi_k(r) \psi_k(\theta)$ for any
$z=r\theta \in \R^{N+1}_+$.  Then thanks to \cite[Lemma
2.1]{Fall-Felli-2014}, 
\eqref{prob-eigenvalue-Ornstein-Uhlenbeck-operator}, and \eqref{prob-eigenvalue-half-a-sphere}, a direct
computation shows that
\begin{equation}\label{prop-eigenvalues:2}
  \varphi_k''(r)+\left(\frac{N+1-2s}{r}-\frac{r}{2}\right)
  \varphi_k'(r)+\left( \gamma-\frac{\nu_k}{r^2}\right)\varphi_k(r)=0,
  \quad \text{ in }  (0,+\infty)
\end{equation}
for any $k\in \mathbb{N} \setminus \{0\}$.
Since $Y \in \mathcal{H}$
\begin{align}\label{prop-eigenvalues:3}
  &+\infty >\int_{\R^{N+1}_+} y^{1-2s} \frac{Y^2}{|z|^2}
    e^{-\frac{|z|^2}{4}}\, dz
    =\int_{0}^\infty r^{N-1-2s} e^{-\frac{r^2}{4}}
    \left(\int_{\mathbb{S}^N_+}\theta_{N+1}^{1-2s}Y^2(r\theta)\, dS\right) \, dr \\
  &=\int_{0}^\infty r^{N-1-2s} e^{-\frac{r^2}{4}}
    \left(\sum_{k=1}^{\infty}\varphi_k^2(r)\right) \, dr\ge
    \int_{0}^\infty r^{N-1-2s} e^{-\frac{r^2}{4}}\varphi_k^2(r) \, dr, \notag 
\end{align}
for any $k\in \mathbb{N} \setminus \{0\}$, thanks to Plancherel's
identity and \eqref{ineq-hardy-extended}.  Analogously
\begin{equation}\label{prop-eigenvalues:3.1}
  +\infty >\int_{\R^{N+1}_+} y^{1-2s} Y^2 e^{-\frac{|z|^2}{4}}\, dz
  \ge \int_{0}^\infty r^{N+1-2s} e^{-\frac{r^2}{4}}\varphi_k^2(r) \, dr, 
\end{equation}
for any $k\in \mathbb{N} \setminus \{0\}$.
Furthermore, letting 
\begin{equation*}
  w_k(t):=(4t)^{\frac{\alpha_k}{2}} \varphi_k(2\sqrt{t})
  \quad\text{ for any } t \in (0,\infty), 
\end{equation*}
 \eqref{prop-eigenvalues:2}  and a direct computation imply that $w_k$ solves 
\begin{equation}\label{prop-eigenvalues:5}
  tw''_k(t) +\left(\frac{N+2-2s}{2}-\alpha_k -t\right)w_k'(t)
  +\left(\frac{\alpha_k}{2}+\gamma\right)w_k(t)\quad\text{ in }  (0,\infty).
\end{equation} 
Equation \eqref{prop-eigenvalues:5} is the well-known Kummer Confluent
Hypergeometric Equation,
\begin{equation}\label{prop-eigenvalues:6}
 tw''_k(t) +\left(b -t\right)w_k'(t)-cw_k(t)\quad\text{ in }  (0,\infty), 
\end{equation}  
with parameters $b= \left(\frac{N+2-2s}{2}-\alpha_k\right)>1$, by
\eqref{ineq-nu-1} and \eqref{def-alphak}, and
$c=-\left(\frac{\alpha_k}{2}+\gamma\right)$, see \cite{AMSI} or
\cite{MD}.  Then the solution $w_k$ can be written as
\begin{equation*}
w_k(t)=A_kM\left(-\frac{\alpha_k}{2}-\gamma,\frac{N+2-2s}{2}-\alpha_k,t\right)+B_k T\left(-\frac{\alpha_k}{2}-\gamma,\frac{N+2-2s}{2}-\alpha_k,t\right)
\end{equation*} 
with $A_k,B_k \in \R$, where $M(c,b,t)$ denotes the Kummer function
and $T(c,b,t)$ denotes the Tricomi function; $M(c,b,t)$ and $T(c,b,t)$
are linearly independent solutions of \eqref{prop-eigenvalues:6} (see
\cite{AMSI} or \cite{MD}).  Furthermore from \cite{AMSI}
\begin{equation}\label{prop-eigenvalues:8}
  T\left(-\frac{\alpha_k}{2}-\gamma,\frac{N+2-2s}{2}-\alpha_k,t\right)
  \sim \co \, t^{1-\frac{N+2-2s}{2}+\alpha_k} \quad \text{ as } t \to 0^+,
\end{equation}
where the constant in \eqref{prop-eigenvalues:8} depends only on
$s,\alpha_k, N,\gamma$ and is different from $0$.  We recall the
following expression for the Kummer function:
\begin{equation*}
M(c,b,t)=\sum_{n=0}^{\infty} \frac{(c)_n}{(b)_n} \frac{t^n}{n!},
\end{equation*}
where $(\cdot)_n$ is the Pochhammer’s symbol defined in
\eqref{def-Pnj}. It is clear that $M(c,b,t)$ has a finite limit as
$t \to 0^+$, while its asymptotic behaviour at $+\infty$ depends on the parameter
$c$.  Then, for any $k \in \mathbb{N} \setminus \{0\}$, if $B_k \neq 0$
\begin{equation*}
w_k(t) \sim \co \, B_k t^{1-\frac{N+2-2s}{2}+\alpha_k} \quad \text{ as } t \to 0^+,
\end{equation*} 
for some $\co\neq0$, and so 
\begin{equation*}
\varphi_k(r) \sim B_k\co \, r^{-N+2s+\alpha_k} \quad \text{ as } r \to 0^+.
\end{equation*}
From \eqref{prop-eigenvalues:3} we deduce that necessarily $B_k=0$ for any
$k \in \mathbb{N} \setminus \{0\}$. Hence
\begin{equation}\label{prop-eigenvalues:12}
w_k(t)=A_kM\left(-\frac{\alpha_k}{2}-\gamma,\frac{N+2-2s}{2}-\alpha_k,t\right).
\end{equation}
Moreover, if $\left(\frac{\alpha_k}{2}+\gamma\right) \notin \mathbb{N}$, then
\begin{equation}\label{prop-eigenvalues:13}
  M\left(-\frac{\alpha_k}{2}-\gamma,\frac{N+2-2s}{2}-\alpha_k,t\right)
  \sim\co \, e^t t^{\frac{\alpha_k}{2}-\gamma -\frac{N}{2}-1+s} \quad \text{ as } t \to +\infty,
\end{equation} 
for some $\co\neq 0$, see \cite{AMSI}. From \eqref{prop-eigenvalues:12} and \eqref{prop-eigenvalues:13} it follows that 
\begin{equation*}
\varphi_k(r) \sim A_k\co \, e{^\frac{r^2}{4}} r^{-2\gamma -N-2+2s} \quad \text{ as } r \to +\infty
\end{equation*} 
and hence necessarily $A_k=0$ for any
$k \in \mathbb{N} \setminus \{0\}$ in view of
\eqref{prop-eigenvalues:3.1}.  In conclusion, if $\gamma$ is an
eigenvalue of \eqref{prob-eigenvalue-Ornstein-Uhlenbeck-operator}, then
there exists $k \in \mathbb{N}\setminus\{0\}$ such that
$(\frac{\alpha_k}{2}+\gamma) \in \mathbb{N}$.

On the other hand, for any $m \in \mathbb{N}$ and
$k \in \mathbb{N}\setminus\{0\}$, letting $Y_{m,k}$ be as in
\eqref{def-egienfunctions}, a direct computation shows that $Y_{m,k}$
is a solution of \eqref{prob-eigenvalue-Ornstein-Uhlenbeck-operator}
with $\gamma:=m-\frac{\alpha_k}{2}$, by \cite[Lemma
2.1]{Fall-Felli-2014}, and
\eqref{prob-eigenvalue-half-a-sphere}. Hence $(m-\frac{\alpha_k}{2})$
is an eigenvalue of
\eqref{prob-eigenvalue-Ornstein-Uhlenbeck-operator}.

From the well-known correspondence between Kummer functions and
 the generalized Laguerre polynomials $L_n^{a}$, we have that
  $P_{j,n}(t)=\binom{n+a_j}{n}^{-1}L_n^{a_j}(t)$, where
  $a_j=\big(\big(\frac{N-2s}{2}\big)^2+\nu_j(\mu)\big)^{1/2}$. Then, recalling
  that $\{\psi_k\}_{k\in \mathbb{N} \setminus \{0\}}$ is an
  orthonormal basis of $L^2(\mathbb{S}_+^N,\theta_{N+1}^{1-2s})$ and
  using the orthogonality relation for Laguerre polynomials, it is
easy to verify that $Y_{m_1,j_1}$ is orthogonal to $Y_{m_2,j_2}$ in
$\mathcal L$ whenever $(m_1,j_1)\neq(m_2,j_2)$. Then we conclude that
\eqref{def-basis-egienfunctions} is an orthonormal basis of
$\mathcal{L}$.
\end{proof}

\begin{proposition} \label{prop-trace-Y-0}
Let $Y$ be a solution  of \eqref{prob-eigenvalue-Ornstein-Uhlenbeck-operator} in the sense of \eqref{eq-eigenvalue-Ornstein-Uhlenbeck-operator} such that $\Tr(Y)=0$. Then $Y\equiv 0$ on $\R_+^{N+1}$.
\end{proposition}
\begin{proof}
If $\Tr(Y)=0$, by
\eqref{prob-eigenvalue-Ornstein-Uhlenbeck-operator} we have that
$\big(-\lim\limits_{y \to 0^+}y^{1-2s}\pd{Y}{y}\big)=0$ on $\R^{N}$. Hence the
function
\begin{equation*}
  \widehat Y(x,y)=
  \begin{cases}
    Y(x,y),&\text{if }y\geq0,\\
    0 ,&\text{if }y<0,
  \end{cases}
\end{equation*}
belongs to  $H^1_{loc}(\R^{N+1},|y|^{1-2s})$ and weakly solves
\begin{equation*}
-\dive(|y|^{1-2s}G\,\nabla \widehat Y)=\gamma |y|^{1-2s}G\,\widehat Y \quad\text{ in } \R^{N+1}.
\end{equation*}
The unique continuation principle for elliptic equations with
Muckenhoupt weights proved in \cite{TXZ} then implies that $\widehat
Y\equiv 0$ in $\R^{N+1}$, so that  $Y\equiv 0$ on $\R_+^{N+1}$.
\end{proof}

\section{An Almgren-Poon type monotonicity
  formula} \label{sec-Almgren-Poon-type-monotonicity-formula} In this
section we develop an Almgren-Poon type monotonicity formula for
solutions of \eqref{eq-weak-formulation-gaussian}. Let $\bar T$ be as
in Proposition \ref{prop-ineq-D+H} with $K=C_g$ and $C_g$ as in  \eqref{hp-g-subhomogeneous}, and
\begin{equation*}
	\alpha:= \frac{T}{2(\lfloor T/\bar T\rfloor+1)}
\end{equation*}
where $\lfloor \cdot \rfloor$ denotes the floor function, i.e. $\lfloor x \rfloor=\max\{j \in \mathbb{Z}:j\le x\}$. It follows that 
\begin{equation*}
	(0,T)=\bigcup_{i=1}^k(a_i,b_i)
\end{equation*}
where
\begin{equation*}
	k=2(\lfloor T/\bar T\rfloor+1)-1, \quad a_i=(i-1) \alpha, \quad \text{ and } \quad b_i=(i+1)\alpha.
\end{equation*}
It is clear that $2 \alpha \in (0, \bar T)$ and
$(a_i,b_i) \cap (a_{i+1},b_{i+1}) \neq \emptyset$. For every $i\in\{1,\dots,k\}$ we define
\begin{align*}
  &V_i(z,t)=U(\sqrt{t} z,t+a_i), \quad z \in \R^{N+1}_+, \quad t \in (0,2 \alpha),\\
  &v_i(x,t)=u(\sqrt{t}x,t+a_i), \quad  x \in \R^N,  \quad t \in (0,2 \alpha),
\end{align*}
see \eqref{def-U} and \eqref{def-u}.
Then  $\Tr(V_i(\cdot,t))=v_i(\cdot,t)$  for every $i=1,\dots,k$ and a.e. $t \in (0,2 \alpha)$.

\begin{remark} \label{reamrk-eq-V-i}
Reasoning as in Section \ref{sec-formulation}, it is easy to see that, for any $i=1 \dots, k$, the function $V_i $ solves 
	\begin{multline}\label{eq:eqVi}
          \sideset{_{\mathcal{H}^*}}{_{\mathcal{H}}}{\mathop{\left\langle
                (V_i)_t,
                \phi\right\rangle}}=\frac{1}{t}\int_{\R^{N+1}_+}y^{1-2s}\nabla
          V_i \cdot \nabla \phi \,G \, dz
          \\-\frac{1}{t}\int_{\R^N}\left(\frac{\mu}{|x|^{2s}}
            v_i(x,t)\phi(x,0)+t^sh(\sqrt{t}x,t+a_i)v_i(x,t)\phi(x,0)\right)G(x,0)\,
          dx,
	\end{multline}
	for any $\phi \in C_c^{\infty}(\overline{\R^{N+1}})$ and a.e. $t \in (0,2 \alpha)$.  
Furthermore,  by  Proposition \ref{prop-eq-Vt} $V_i  \in L^2((\tau,2\alpha),\mathcal{H})$   and $(V_i)_t \in L^2((\tau,2\alpha),\mathcal{H}^*)$  for any $\tau \in (0,2 \alpha)$.
\end{remark}

For any $i=i, \dots, k$ and  $t \in (0,2 \alpha)$, let 
\begin{equation*}
H_i(t):=\int_{\R^{N+1}_+} y^{1-2s}V_i^2G \, dz 
\end{equation*}
and 
\begin{equation}\label{def-D}
D_i(t):=\frac{1}{t}\int_{\R^{N+1}_+} y^{1-2s}|\nabla V_i|^2 G \, dz-\frac{1}{t}\int_{\R^N}  \left(\frac{\mu}{|x|^{2s}}v_i^2 +t^sh(\sqrt{t}x,t+a_i) v_i^2\right)G(x,0)\, dx.
\end{equation}

\begin{proposition}\label{prop-drivitive-H}
For any $i=1,\dots,k$, we have that $H_i \in W^{1,1}_{loc}(0,2 \alpha )$ and 
\begin{equation}\label{eq-H'}
H_i'(t)=2\df{\mathcal{H}^*_t}{(V_i)_t}{V_i}{_{\mathcal{H}_t}}=2 D_i(t) 
\end{equation}
in a distributional sense and  a.e. in $(0,2 \alpha)$.
\end{proposition}

\begin{proof}
The claim follows from Remark \ref{remark-parabolic-change-variables}  and  \eqref{def-D}.
\end{proof}

\begin{proposition}\label{prop-Ht-monotonicity}
Let $C_{N,s,\mu}$ be as in Proposition \ref{prop-ineq-D+H} with $K:= C_g$ and $C_g$ as in  \eqref{hp-g-subhomogeneous}. Then the function
\begin{equation*}
t \to t^{-2C_{N,s,\mu}+\frac{N-2+2s}{2}}H_i(t)
\end{equation*}
is non-decreasing in $(0,2\alpha)$.
\end{proposition}

\begin{proof}
 In view of \eqref{ineq-D+H} and \eqref{eq-H'}
\begin{equation*}
H_i'(t) \ge \frac{1}{t} \left(2C_{N,s,\mu}-\frac{N-2+2s}{2}\right)H_i(t) \quad \text{for a.e.  } t \in (0,2 \alpha),
\end{equation*}
hence 
\begin{equation*}
  \frac{d}{dt}\left(t^{-2C_{N,s,\mu}+\frac{N-2+2s}{2}}H_i(t)\right) \ge 0  \quad \text{for a.e.  } t \in (0,2 \alpha).
\end{equation*}
We conclude that $t \to t^{-2C_{N,s,\mu}+\frac{N-2+2s}{2}}H_i(t)$ is non-decreasing in $(0,2 \alpha)$.
\end{proof}

\begin{corollary}
If $1\le i\le k$ and  $H_i(\bar t)=0$ for some $ \bar t \in (0,2 \alpha)$, then $H_i(t)=0$ for any $t \in (0,\bar t)$.
\end{corollary}
\begin{proof}
  Since $t \to t^{-2C_{N,s,\mu}+\frac{N-2+2s}{2}}H_i(t)$ is
  non-decreasing in $(0,2 \alpha)$ by Proposition
  \ref{prop-Ht-monotonicity} and it is non-negative, from
  the assumption $H(\bar t)=0$ it follows that $H_i(t)=0$ for any $t \in (0,\bar t)$.
\end{proof}
The regularity of the function $tD_i(t)$ is discussed in the following proposition.
\begin{proposition}\label{prop-derivative-tD}
If $1 \le i \le k$ and $T_i \in (0,2 \alpha)$ is such that $V_i(\cdot,T_i) \in \mathcal{H}$ then 
\begin{itemize}
\item[(i)] $(V_i)_t \in L^2((\tau,T_i),\mathcal{L})$ for any $\tau \in (0,T_i)$, 
\item[(ii)] the function  $t \to tD_i(t)$ belongs to $W^{1,1}_{loc}(0, T_i)$ and its weak derivative is  as follows:
\begin{multline}\label{eq-derivative-tD}
\frac{d}{dt} (tD_i(t))=2t \int_{\R_+^{N+1}}y^{1-2s}\left|(V_i)_t\right|^2 Gdz\\
-\int_{\R^N}\left(st^{s-1}h(\sqrt{t}x,t+a_i)+t^{s-\frac12}\nabla h(\sqrt{t}x,t+a_i)\cdot \frac{x}{2}+t^s h_t(\sqrt{t}x,t+a_i) \right) v_i^2(x,t) G(x,0) \, dx.
\end{multline}
\end{itemize}
\end{proposition}

\begin{proof}
Let $1\le i\le k$. Then (i) follows from Proposition   \ref{prop-Vt-L2} and Remark \ref{reamrk-eq-V-i}.

With an approximating procedure similar to Proposition
\ref{prop-Vt-L2}, formally testing \eqref{eq:eqVi}  with $(V_i)_t$
yields, for a.e. $\tau \in (0,T_i)$,
\begin{align*}
  &\int_{\R^{N+1}_+}y^{1-2s}|\nabla V_i(\cdot,\tau)|^2G \, dz-
    \int_{\R^N}\left(\frac{\mu}{|x|^{2s}}v_i^2(\cdot,\tau)+\tau^{s}
    h(\sqrt{\tau}x,\tau+a_i)
    v_i^2(\cdot,\tau) \right)G(x,0) \, dx \\
  &=\int_{\R^{N+1}_+}y^{1-2s}|\nabla V_i(\cdot,T_i)|^2G \, dz\\
  &-\int_{\R^N}\left(\frac{\mu}{|x|^{2s}}v_i^2(\cdot,T_i)+ {T_i}^{s}
    h(\sqrt{T_i}x,T_i+a_i)
    v_i^2(\cdot,T_i) \right)G(x,0) \,  dx\\
  &-2\int_{\tau}^{T_i}t\left(\int_{\R^{N+1}}y^{1-2s} (V_i)_t^2 G \, dz \right)\, dt\\
  &+\int_{\tau}^{T_i}\left(\int_{\R^N}\left(st^{s-1}h(\sqrt{t}x,t+a_i)+t^s
    h_t(\sqrt{t}x,t+a_i)\right)
    v_i^2(x,t) G(x,0) \, dx \right)\, dt\\
  &+\int_{\tau}^{T_i}\left(\int_{\R^N}t^{s-\frac12}\nabla
    h(\sqrt{t}x,t+a_i)\cdot \frac{x}{2}
    v_i^2(x,t) G(x,0) \, dx \right)\, dt,
\end{align*}
hence, thanks to \eqref{hp-g-Lr}, \eqref{hp-g-subhomogeneous} and \eqref{def-h}, the function 
\begin{equation*}
  \tau \mapsto \int_{\R^{N+1}_+}y^{1-2s}|\nabla V_i(\cdot,\tau)|^2G \,
  dz-
  \int_{\R^N}\left(\frac{\mu}{|x|^{2s}}v_i^2(\cdot,\tau)+\tau^{s} h(\sqrt{\tau}x,\tau)v^2(\cdot,\tau) \right)G(x,0) \, dx 
\end{equation*}
is absolute continuous on $[T_1,T_2]$ for any $[T_1,T_2] \subset (0,T_i)$ and, for a.e. $\tau \in (0,T_i)$,
\begin{multline*}
  \frac{d}{d\tau}\Bigg(\int_{\R^{N+1}_+}y^{1-2s}|\nabla V_i(\cdot,\tau)|^2G \, dz\\
  -\int_{\R^N}\left(\frac{\mu}{|x|^{2s}}v_i^2(\cdot,\tau)+\tau^{s}
    h(\sqrt{\tau}x,\tau+a_i)
    v_i^2(\cdot,\tau) \right)G(x,0) \, dx\Bigg) \\
  =2\tau\int_{\R^{N+1}_+}y^{1-2s} (V_i)_t^2(z,\tau) G(z) \,
  dz-\int_{\R^N}\tau^{s-\frac12}
  \nabla h(\sqrt{\tau}x,\tau+a_i)\cdot \frac{x}{2}v_i^2(\tau,x) G(x,0) \, dx \\
  -\int_{\R^N}\left(s\tau^{s-1}h(\sqrt{\tau}x,\tau+a_i)+\tau^s
    h_t(\sqrt{\tau}x,\tau+a_i)\right) v_i^2(\tau,x) G(x,0) \, dx.
\end{multline*}
The proof is thereby complete.
\end{proof}

For any $i=1 \dots k$, let us define the Almgren-Poon frequency function  
\begin{equation*}
\mathcal{N}_i:(0,2 \alpha) \to \R \cup \{-\infty,+\infty\}, \quad \mathcal{N}(t):=\frac{tD_i(t)}{H_i(t)}.
\end{equation*}

\begin{proposition} \label{prop-N-W11}
If there exists $\beta_i, T_i \in (0,2 \alpha)$ such that 
\begin{equation}\label{betai-ti-properties}
\beta_i <T_i, \quad H_i(t) >0 \text{ for all } t \in (\beta_i,T_i), \text { and } V_i(\cdot,T_i)\in \mathcal{H},
\end{equation} 
then $\mathcal{N}_i\in W_{loc}^{1,1}(\beta_i,T_i)$ and the weak derivative of $\mathcal{N}_i$ can be written as 
\begin{equation}\label{eq-derivative-N}
\mathcal{N}_i'(t)=\nu_{1,i}(t) +\nu_{2,i}(t) \quad \text{ for a.e. } t \in (\beta_i,T_i)
\end{equation}
where
\begin{align}\label{def-nu1}
\nu_{1,i}(t):=\frac{2t}{H_i^2(t)}\Bigg[\left(\int_{\R_+^{N+1}}y^{1-2s}\left|(V_i)_t\right|^2Gdz \right)&
  \left(\int_{\R^{N+1}_+}y^{1-2s}V_i^2 G \,dz\right)\\
  &
-\left(\int_{\R_+^{N+1}}y^{1-2s}(V_i)_tV_iGdz\right)^2\Bigg]\notag
\end{align} 
and 
\begin{align}\label{def-nu2}
  \nu_{2,i}(t):=-\frac{1}{H_i(t)}   \times
                                   \Bigg(\int_{\R^N}\bigg(st^{s-1}h(\sqrt{t}x,t+a_i)
                                   &+t^{s-\frac12}\nabla h(\sqrt{t}x,t+a_i)\cdot \frac{x}{2}\\
                                 &+t^s h_t(\sqrt{t}x,t+a_i) \bigg) v_i^2(x,t) G(x,0) \, dx\Bigg).\notag
\end{align}
Furthermore $\nu_{1,i}(t) \ge 0$ for a.e. $t \in (\beta_i,T_i)$ and 
\begin{equation}\label{ineq-N-bounded-below}
\mathcal{N}_i(t) > -\frac{N+2-2s}{4} \quad \text{ for any } t \in (\beta_i,T_i).
\end{equation}
\begin{proof}
  Since $H_i(t)>0$ for any $t \in (\beta_i,T_i)$, then
  $1/{H_i},tD_i \in W^{1,1}_{loc}(\beta_i,T_i)$ by Proposition
  \ref{prop-drivitive-H} and Proposition
  \ref{prop-derivative-tD}. Hence
  $\mathcal{N}_i \in W^{1,1}_{loc}(\beta_i,T_i)$. Furthermore
\begin{equation*}
\mathcal{N}_i'(t)=\frac{(tD_i)'H_i-tD_iH_i'}{H_i^2}
\end{equation*}
and so, thanks to  \eqref{eq-H'} and \eqref{eq-derivative-tD}, we
conclude that \eqref{eq-derivative-N} holds with $\nu_{1,i}$ and $\nu_{2,i}$
as in \eqref{def-nu1} and \eqref{def-nu2} respectively.  By
the Cauchy-Schwarz inequality in $\mathcal{L}$ we have
$\nu_{1,i}(t) \ge 0$ for a.e. $t \in (\beta_i,T_i)$. Finally, 
\eqref{ineq-N-bounded-below} follows directly from \eqref{ineq-D+H}
with the function $f(x,t):=h(x,t+a_i)$.
\end{proof}
\end{proposition}

The remainder term $\nu_{2,i}$ can be estimated in terms of the
frequency function as follows.
\begin{proposition}\label{prop-nu2-estimates}
Let $\nu_{2,i}$ be as in \eqref{def-nu2}. Then there exists a constant $C_1>0$ such that, if $i \in \{i,\dots,k\}$ and $\beta_i,T_i \in (0,2 \alpha )$ are as in \eqref{betai-ti-properties}, then  
\begin{equation}\label{ineq-nu2-estimates}
|\nu_{2,i}(t)| \le 	C_1 \left(t^{-1+\frac{\e}{2}}+\norm{ h_t(\cdot,t+a_i)}_{L^{\frac{N}{2s}}(\R^N)}\right)\left(\mathcal{N}_i(t)+\frac{N+2-2s}{4}\right)
\end{equation}
for a.e. $t \in (\beta_i,T_i)$.
\end{proposition}

\begin{proof}
 Estimate \eqref{ineq-nu2-estimates} follows from
  \eqref{ineq-D+H}, \eqref{ineq-h}, and \eqref{ineq-h_t}, taking into
  account that, by a
  change of variables,
\begin{equation*}
  \norm{t^sh_t(\sqrt{t}\cdot,t+a_i)}_{L^{\frac{N}{2s}}(\R^N)}
  =\norm{h_t(\cdot,t+a_i)}_{L^{\frac{N}{2s}}(\R^N)}
\end{equation*}
for any $i \in \{i,\dots,k\}$.
\end{proof}

\begin{proposition} \label{prop-N-bounded} There exists a constant
  $C_2>0$ such that, if $i \in \{i,\dots,k\}$ and
  $\beta_i,T_i \in (0,2 \alpha )$ are as in
  \eqref{betai-ti-properties}, then
\begin{equation}\label{ineq-N-estimates}
\mathcal{N}_i(t) \le -\frac{N+2-2s}{4}+C_2\left(\mathcal{N}_i(T_i)+\frac{N+2-2s}{4}\right)
\end{equation}
for  any $t \in (\beta_i,T_i)$.
\end{proposition}

\begin{proof}
  Since $\nu_{1,i}\ge0$ by Proposition \ref{prop-N-W11}, from
  \eqref{ineq-nu2-estimates} it follows that, a.e. in $(\beta_i,T_i)$,
\begin{equation*}
\mathcal{N}'_i(t) \ge-C_1 \left(t^{-1+\frac{\e}{2}}+\norm{h_t(\cdot,t+a_i)}_{L^{\frac{N}{2s}}(\R^N)}\right)\left(\mathcal{N}_i(t)+\frac{N+2-2s}{4}\right). 
\end{equation*} 
By integration we obtain the estimate
\begin{equation*}
  \mathcal{N}_i(t) \le-\frac{N+2-2s}{4}+\left(\mathcal{N}_i(T_i)
    +\frac{N+2-2s}{4}\right)
  e^{\big(\frac{2C_1}{\epsilon}T_i^{\e/2}
    +C_1\norm{h_t(\cdot,t+a_i)}_{L^1((0,2 \alpha ),L^{N/(2s)}(\R^N))} \big)}
\end{equation*}
for any  $t \in (\beta_i,T_i)$, which implies
\eqref{ineq-N-estimates} in view of \eqref{hp-g-Lr} and \eqref{def-h}.
\end{proof}

\begin{proposition}\label{prop-H-positive}
For any $i \in \{i, \dots,k\}$, if $H_i(t)\not \equiv 0$ then 
\begin{equation}\label{ineq-H-positive}
H_i(t)>0 \quad \text{ for all } t \in (0,2 \alpha).
\end{equation}
\end{proposition}

\begin{proof}
  Since $H_i(t)\not \equiv 0$ and $H_i$ is continuous by Remark
  \ref{remark-parabolic-change-variables}, there exists
  $T_i\in (0,2 \alpha)$ such that
\begin{equation}\label{H-Ti-positive}
	H_i(T_i)>0 \quad \text{ and } \quad V_i(\cdot,T_i)\in \mathcal{H}.
\end{equation}
By Proposition \ref{prop-Ht-monotonicity} it follows that $H_i(t)>0$
for any $ t \in [T_i,2 \alpha)$. If we define
\begin{equation*}
t_i:=\inf\{\tau \in (0,T_i):H_i(t)>0 \text{ for all } t \in (\tau,2 \alpha)\},
\end{equation*}	
then either 
\begin{equation*}
t_i=0 \quad \text{ and } \quad H_i(t)>0 \text{ for all } t \in (0,2 \alpha) 
\end{equation*}
or 
\begin{equation}\label{H-bad-case}
0<t_i<T_i \quad \text{ and }  \quad 
\begin{cases}
H_i(t)=0, \text{ for any } &t \in (0,t_i],\\
H_i(t)>0, \text{ for any } &t \in (t_i,2 \alpha).\\
\end{cases}
\end{equation}
Now we prove that the second case can not occur arguing by
contradiction. If \eqref{H-bad-case} holds, then, thanks to
Proposition \ref{prop-N-bounded} and \eqref{eq-H'},
\begin{equation*}
\dfrac{t}{2}H_i'(t)\le \left(-\frac{N+2-2s}{4}+C_2\left(\mathcal{N}_i(T_i)+\frac{N+2-2s}{4}\right)\right) H_i(t) 
\end{equation*}
for a.e. $t \in (t_i,T_i)$. Integrating the above inequality  we obtain
\begin{equation*}
  H_i(t)\ge \frac{t^{2\left(-\frac{N+2-2s}{4}+
        C_2\left(\mathcal{N}_i(T_i)
          +\frac{N+2-2s}{4}\right)\right)}}
  {T_i^{2\left(-\frac{N+2-2s}{4}+C_2\left(\mathcal{N}_i(T_i)+\frac{N+2-2s}{4}\right)\right)}} H_i(T_i)
\end{equation*}
for all $t \in [t_i,T_i)$.  Since $H_i(t_i)=0$, we have reached a
contradiction in view of \eqref{H-Ti-positive}.  In conclusion,
\eqref{ineq-H-positive} must hold.
\end{proof}

\begin{proposition}\label{prop-Hi-not-0-then-Hi+1-not-0}
For any $i \in \{1,\dots,k-1\}$ 
\begin{equation*}
H_i(t)\equiv 0  \text{ in } (0,2\alpha) \text{ if and only if } H_{i+1}(t) \equiv 0 \text{ in } (0,2\alpha).
\end{equation*}
\end{proposition}

\begin{proof}
  We start by proving that if $H_i(t)\equiv 0 \text{ in } (0,2\alpha)$
  then $H_{i+1}(t) \equiv 0 \text{ in } (0,2\alpha)$. By
  contradiction, if there exists $\bar t \in (0,2\alpha)$ such that
  $H_{i+1}(\bar t)>0$, then $H_{i+1}(t)>0$ for all $t \in (0,2\alpha)$
  by Proposition \ref{prop-H-positive}. It follows that
  $V_{i+1}(\cdot,t) \not \equiv 0$ for all $t \in (0,2\alpha)$ and
  $V(\cdot,t) \not \equiv 0$ for all $t \in (i
  \alpha,(i+1)\alpha)$. Therefore $V_i(\cdot,t) \not \equiv 0$ for
  some $ t \in (0,2\alpha)$, which is a contradiction.

  Now let us prove that, if
  $H_{i+1}(t)\equiv 0 \text{ in } (0,2\alpha)$, then
  $H_i(t) \equiv 0 \text{ in } (0,2\alpha)$. By contradiction, let us
  assume that $H_i(t)\not \equiv 0$.  Then $H_i(t)>0$ for any
  $t \in (0,2\alpha)$ by Proposition
  \ref{prop-H-positive}. It follows that $V_i(\cdot,t)\not \equiv 0$
  for all $t \in (0,2 \alpha)$ and so
  $V_{i+1}(\cdot,t) \not \equiv 0$ for all $t \in (0,\alpha)$,
  hence $H_{i+1}(t)\not \equiv 0$ for all $t\in(0,\alpha)$,
  which is a contradiction.
\end{proof}

\begin{proposition}\label{prop-H-positve-everywhere}
  If $U$ is a weak solution of \eqref{prob-frac-extended} such that
  $U \not \equiv 0$ in $\R_+^{N+1} \times (0,T)$, then
\begin{equation*}
  H_i(t)>0 \quad \text{for any } t \in (0,2 \alpha) \text{ and } i \in \{1,\dots,k\}.
\end{equation*}
\end{proposition}

\begin{proof}
  If $U \not \equiv 0$ in $\R_+^{N+1} \times (0,T)$, then there exists
  some $i_0 \in \{1,\dots,k\}$ such that $V_i \not \equiv 0$ in
  $(0,2 \alpha)$. Then $H_{i_0}(t) \not \equiv 0$ in $(0,2
  \alpha)$. Thanks to Proposition \ref{prop-Hi-not-0-then-Hi+1-not-0},
  $H_i(t) \not \equiv 0$ in $(0,2 \alpha)$ for any
  $i \in \{1,\dots,k\}$. In view of
  Proposition \ref{prop-H-positive}, we can therefore conclude that $H_i(t)>0$ for any
  $t \in (0,2 \alpha)$ and   $i \in \{1,\dots,k\}$.
\end{proof}

\begin{proof}[\textbf{Proof of Theorem \ref{theorem-unicity-Cauchy}}]
  It is not restrictive to assume that $t_0=0$. Let $W$ be a solution
  of \eqref{eq-weak-frac-extended-forward}. Let $\bar t \in (0,T)$ be
  such that $W(z,-\bar t)\equiv0$ in $\R^{N+1}_+$, {so that, letting $U$ be
  as in \eqref{def-U}, $U(z,\bar t)\equiv 0$ in $\R^{N+1}_+$}. Then
  $\bar t \in (a_i,b_i)$ for some $i \in \{1,\dots,k\}$ and
  $H_i(\bar t-a_i)=0$.  By Proposition
  \ref{prop-H-positve-everywhere} it follows that $U\equiv 0$ in
  $\R^{N+1}_+\times (0,T)$ and hence, by \eqref{def-U}, $W \equiv 0$ in
  $\R^{N+1}_+\times (-T,0)$.
\end{proof}

From now on, we assume that $U \not \equiv 0$ in $\R^{N+1}_+\times
(0,T)$ and, defining $V$ as in
\eqref{def-V}, we denote, for all $t\in(0,2\alpha)$,
\begin{align*}
  &H(t):=H_1(t)=\int_{\R^{N+1}_+} y^{1-2s}V^2(z,t)G(z) \, dz,	\\
  &D(t):=D_1(t)=\frac{1}{t}\bigg(\int_{\R^{N+1}_+} y^{1-2s}|\nabla
    V|^2G \, dz
    -\int_{\R^N}  \left(\frac{\mu}{|x|^{2s}}v^2
    +t^sh(\sqrt{t}x,t) v^2\right)G(x,0)\, dx\bigg).	
\end{align*}
Since we are assuming that $U \not \equiv 0$ in
$\R^{N+1}_+\times (0,T)$, thanks to Proposition
\ref{prop-H-positve-everywhere} the Almgren-Poon type frequency
function
\begin{equation*}
\mathcal{N}:(0,2 \alpha)\to \R, \quad \mathcal{N}(t):=\frac{tD(t)}{H(t)}.
\end{equation*}
is well-defined. Furthermore, in view of Proposition \ref{prop-N-W11}
$\mathcal{N} \in W^{1,1}_{loc}(0,2 \alpha)$  and 
\begin{equation*}
\mathcal{N}'(t)=\nu_1(t)+\nu_2(t) \quad \text{ for a.e. } t \in (0,2 \alpha),
\end{equation*} 
where we have defined 
\begin{equation}\label{def-nu12}
\nu_1(t):=\nu_{1,1}(t), \quad \nu_2(t):=\nu_{2,1}(t),
\end{equation}
according to   notation \eqref{def-nu1}--\eqref{def-nu2}.
Since $V(\cdot,t) \in \mathcal{H}$ for a.e. $t \in (0,T)$, there exists
\begin{equation}\label{def-T0}
T_0 \in (0,2 \alpha) \quad \text{ such that } \quad V(\cdot,T_0) \in \mathcal{H}. 
\end{equation}

\begin{proposition}\label{prop-limit-N} The limit 
\begin{equation}\label{limit-N}
\gamma:=\lim_{t \to 0^+}	\mathcal{N}(t)
\end{equation}
exists and it is finite.
\end{proposition}
\begin{proof}
  From Propositions \ref{prop-N-W11} and 
  \ref{prop-N-bounded} it follows that $\mathcal{N}$ is
  bounded. Hence the limit \eqref{limit-N}, if it exists, is
  finite.  Furthermore, from Proposition \ref{prop-N-W11} we have that
  $\nu_1 \ge 0$ a.e. in $(0,2 \alpha)$, whereas
  $\nu_ 2 \in L^1(0,T_0)$ by Proposition \ref{prop-nu2-estimates}, 
  Proposition \ref{prop-N-bounded}, \eqref{def-h}, and
    \eqref{hp-g-Lr}, where $T_0$ is as in
  \eqref{def-T0}.  Then, from
\begin{equation*}
\mathcal{N}(t)=\mathcal{N}(T_0)-\int_t^{T_0} \mathcal{N}'(\tau) \, d\tau=\mathcal{N}(T_0)-\int_t^{T_0} \nu_1(\tau) \, d\tau-\int_t^{T_0} \nu_2(\tau) \, d\tau,
\end{equation*}
we conclude that the limit \eqref{limit-N} exists.
\end{proof}

\begin{proposition}\label{prop-H-estmiates}
Let $T_0$ be as in \eqref{def-T0} and $\gamma$ as in \eqref{limit-N}. Then there exists a constant $K_1>0$ such that 
\begin{equation}\label{ineq-H-estimates-above}
H(t) \le K_1 t^{2 \gamma} \quad\text{for all t } \in (0,T_0).
\end{equation}
Moreover, for any $\sigma>0$, there exists a constant $K_2(\sigma)>0$
such that
\begin{equation}\label{ineq-H-estimates-below}
H(t) \ge K_2(\sigma) t^{2 \gamma+\sigma} \quad\text{for all } t\in (0,T_0).
\end{equation}
\end{proposition}
\begin{proof}
  Thanks to the H\"older inequality, \eqref{def-h}, \eqref{hp-g-Lr},
  Proposition \ref{prop-N-W11}, Proposition \ref{prop-nu2-estimates},
  and Proposition \ref{prop-N-bounded}, for any $t\in (0,T_0)$ we have
  that
\begin{multline*}
  \mathcal{N}(t) -\gamma=\int_{0}^{t}(\nu_1(\tau)+\nu_2(\tau)) \,
  d\tau
  \ge  \int_{0}^{t}\nu_2(\tau) \, d\tau \\
  \ge - C_1C_2
  \left(\mathcal{N}_i(T_i)+\frac{N+2-2s}{4}\right)
  \int_{0}^{t}\left(\tau^{-1+\frac{\e}{2}}
    +\norm{h_t(\cdot,\tau)}_{L^{\frac{N}{2s}}(\R^N)}\right)
  \, d\tau \ge- C_3 t^{\delta},
\end{multline*}
for some constant $C_3>0$, where
$\delta:=\min\left\{\frac{\e}{2},1-\frac{1}{r}\right\}$.  It follows
that, taking into account \eqref{eq-H'},
\begin{equation*}
  \frac{H'(t)}{H(t)}= \frac{2}{t}\mathcal{N}(t)
  \ge \frac{2\gamma}{t}-2C_3 t^{-1+\delta}\quad \text{for a.e. } t \in (0,T_0).
\end{equation*}
An integration over $(t,T_0)$ yields
\begin{equation*}
  H(t)\le \frac{H(T_0)}{T_0^{2
      \gamma}}e^{\frac{2C_3T_0^\delta}{\delta}}t^{2 \gamma}  \quad
  \text{for all } t\in (0,T_0),
\end{equation*} 
so that \eqref{ineq-H-estimates-above} is proved.

Furthermore, since $\gamma:=\lim_{t \to 0^+}\mathcal{N}(t)$, for any
$ \sigma >0$ there exists $T_\sigma \in (0,T_0)$ such that
$\mathcal{N}(t)<\gamma +\sigma /2$ for any $t \in (0,T_\sigma)$, and hence
\begin{equation*}
\frac{H'(t)}{H(t)}= \frac{2}{t}\mathcal{N}(t) <\frac{2\gamma+\sigma}{t} \quad \text{ for any } t \in (0,T_\sigma).
\end{equation*}
An integration of the above estimate over $(t,T_\sigma)$,
  together with continuity and positivity of $H$ in $[T_\sigma,T_0]$,
  yields \eqref{ineq-H-estimates-below} for some constant
$K_2(\sigma)$.
\end{proof}

\section{The Blow-up Analysis} \label{sec-blow-up}

If $V$ is a solution of \eqref{eq-weak-formulation-gaussian}, then it
is easy to see that, for any $\la>0$, the function
\begin{equation}\label{def-V-la}
V_\la(z,t):=V( z,\la^2t)  
\end{equation}
belongs to $L^2((\tau,T/ \la^2), \mathcal{H})$ for all
  $0<\tau<T/\la^2$ and  solves 
\begin{align}
&\sideset{_{\mathcal{H}^*}}{_{\mathcal{H}}}{\mathop{\left\langle
                (V_\la)_t, \phi\right\rangle}}
                =\frac{1}{t}\int_{\R^{N+1}_+}y^{1-2s}\nabla V_\la
                \cdot \nabla \phi \,G \, dz
                \label{eq-weak-formulation-gaussian-Vla}\\
  &-\frac{1}{t}\int_{\R^N}\left(\frac{\mu}{|x|^{2s}} v_\la(x,t)\phi(x,0)
    +\la^{2s}t^s h(\la \sqrt{t}x,\la^2t)v_\la(x,t)\phi(x,0)\right)G(x,0)\, dx,\notag
\end{align}
for a.e. $t \in (0,T/\la^2)$ and any $\phi \in C^{\infty}_c(\overline{\R^{N+1}_+})$, where
\begin{equation*}
v_\la(x,t):=v(x,\la^2t)=\Tr(V_\la(\cdot,t))(x). 
\end{equation*}
We can also define the height and  energy functions
associated to the scaled equation
\eqref{eq-weak-formulation-gaussian-Vla} as
\begin{align}
  H_\la(t)&=\int_{\R^{N+1}_+} y^{1-2s}V_\la^2 G \, dz, \notag 	\\
  D_\la(t)&=\frac{1}{t}\left(\int_{\R^{N+1}_+} y^{1-2s}|\nabla V_\la|^2G
            \, dz
            -\int_{\R^N}  \left(\frac{\mu}{|x|^{2s}}v_\la^2
            +\la^{2s}t^sh(\la\sqrt{t}x,\lambda^2t) v_\la^2\right)G(x,0)\, dx\right)\notag
\end{align}
and the Almgren-Poon frequency function as
\begin{equation}\label{def-N-scaled}
\mathcal{N}_\la(t):=\frac{tD_\la(t)}{H_\la(t)}.	
\end{equation}
For any $\la>0$, we have the following scaling properties:
\begin{equation}\label{D-H-N-scaling}
  D_\la(t)=\la^2D(\la^2t), \quad H_\la(t)= H(\la^2t) \quad \text{ and
  }
  \quad \mathcal{N}_\la(t)=\frac{\la^2tD(\la^2t)}{H(\la^2t)}=\mathcal{N}(\la^2 t),
\end{equation}
on $(0,2\alpha/\la^2)$. Let, for any $\la \in (0,\sqrt{T_0})$ and for any $t \in (0,T/T_0)$,
\begin{equation}\label{def-W-la}
W_\la(z,t):=\frac{V(z,\la^2t)}{\sqrt{H(\la^2)}}.
\end{equation}
In particular we note that $ 1  \in (0,T/T_0)$. Similarly, we define  
\begin{equation*}
  w_\la(z,t):=\frac{v(z,\la^2t)}{\sqrt{H(\la^2)}}
\end{equation*}
for any $t \in (0,T/T_0)$.
From \eqref{def-V-la} and \eqref{def-W-la} we deduce that $W_\la$
belongs to $L^2((\tau,T/ \la^2), \mathcal{H})$ for all
  $0<\tau<T/\la^2$ and solves \eqref{eq-weak-formulation-gaussian-Vla}.

\begin{proposition}\label{prop-Wla-Wlat-bounded}
Let $T_0$ be as in \eqref{def-T0}. Then 
\begin{equation}\label{Wla-bounded}
  \{W_\la\}_{\la \in (0,\sqrt{T_0})}\text{ is bounded in  }
  L^\infty((\tau,1),\mathcal{H})
  \quad \text{ for any } \tau \in (0,1),
\end{equation}
and 
\begin{equation}\label{Wla-t-bounded}
  \{(W_\la)_t\}_{\la \in (0,\sqrt{T_0})}\text{ is bounded in  }
  L^\infty((\tau,1),\mathcal{H}^*)  \quad \text{ for any } \tau \in (0,1).
\end{equation}
Moreover
\begin{equation}\label{Wla-compact}
  \{W_\la\}_{\la \in (0,\sqrt{T_0})}\text{ is relatively compact in }
  C^0([\tau,1],\mathcal{L}) \quad \text{ for any } \tau \in (0,1).
\end{equation}
\end{proposition}

\begin{proof}
From Proposition \ref{prop-Ht-monotonicity}, for any $t \in (0,1)$,
\begin{equation}\label{eq-prop-Wla-Wlat-bounded:1}
  \int_{\R^{N+1}_+} y^{1-2s} W^2_\la(z,t) \, G(z) \, dz
  = \frac{H(\la^2t)}{H(\la^2)} \le t^{2C_{N,s,+\mu}-\frac{N-2+2s}{2}}.
\end{equation}
Furthermore, by \eqref{ineq-D+H}, \eqref{ineq-N-estimates}, and \eqref{D-H-N-scaling}
\begin{multline*}
  \frac{1}{t} \left(-\frac{N+2-2s}{4}+C_2\left(\mathcal{N}(T_0)
      +\frac{N+2-2s}{4}\right)\right)H_\la(t)\ge \la^2D(\la^2t)\\
  \ge \frac{1}{t}\left(-\frac{N+2-2s}{4}+C_{N,s,\mu} \right)H_\la(t)
  +\frac{1}{t}C_{N,s,\mu} \int_{\R^{N+1}_+} y^{1-2s} |\nabla
  V_\la(z,t)|^2 G(z) \, dz, \end{multline*} for all
$ \la \in (0,\sqrt{T_0})$ and a.e. $ t \in (0,1)$. It follows that,
taking into account \eqref{eq-prop-Wla-Wlat-bounded:1},
\begin{align}\label{eq-prop-Wla-Wlat-bounded:2}
  \int_{\R^{N+1}_+} &y^{1-2s} |\nabla V_\la(z,t)|^2 G(z) \, dz  \\
                    &\notag\le
                      C^{-1}_{N,s,h}\left(C_2\left(\mathcal{N}(T_0)+\frac{N+2-2s}{4}\right)
  -C_{N,s,\mu} \right) \int_{\R^{N+1}_+} y^{1-2s} V^2_\la(z,t) G(z) \, dz\\
&\notag\le
C^{-1}_{N,s,h}\left(C_2\left(\mathcal{N}(T_0)+\frac{N+2-2s}{4}\right)
  -C_{N,s,\mu} \right) t^{2C_{N,s,\mu}-\frac{N-2+2s}{2}}H(\la^2)
\end{align}
for a.e. $t \in (0,1)$. Hence we have proved \eqref{Wla-bounded} in
view of \eqref{eq-prop-Wla-Wlat-bounded:1} and
\eqref{eq-prop-Wla-Wlat-bounded:2}.

To prove \eqref{Wla-t-bounded}, we note that, from \eqref{def-W-la}, 
\begin{equation*}
(W_\la)_t=\frac{\la^2 }{\sqrt{H(\la^2)}}V_t(z,\la^2 t ).
\end{equation*}
To estimate $\norm{(W_\la)_t(\cdot,t)}_{\mathcal{H}^*}$, we observe that, for any $\phi\in \mathcal{H}$,
\begin{align}\label{eq-prop-Wla-Wlat-bounded:4}
 & \la^{2s}t^{-1+s}\left|\int_{\R^N}
    h(\la\sqrt{t}x,\la^2t)w_\la(x,t)\Tr\phi(x)G(x,0)\, dx
  \right|\\
 \notag&\le C_g\la^{2s}t^{-1+s}\int_{\R^N}|w_\la(x,t)||\Tr\phi(x)|G(x,0)\, dx\\
 \notag &\quad+C_g\la^\e t^{-1+\e/2} \int_{\R^N}|x|^{-2s+\e}|w_\la(x,t)||\Tr\phi(x)|G(x,0)\, dx\\
\notag  &\le
  C_gK_{N,s}\la^{2s}t^{-1+s}\norm{W_\la(\cdot,t)}_{\mathcal{H}}
  \norm{\phi}_{\mathcal{H}}\\
\notag&\quad  +C_g\la^{\e}t^{-1+\e/2}\int_{\{|x| \le 1\}}|x|^{-2s}|w_\la(x,t)||\Tr\phi(x)|G(x,0)\, dx\\
  \notag&\quad+C_g\la^{\e}t^{-1+\e/2}\int_{\{|x| \ge 1\}}|w_\la(x,t)||\Tr\phi(x)|G(x,0)\, dx \\
\notag  &\le \frac{C_g\,\la^\e}{
  t^{1-\e/2}}\left(K_{N,s}t^{s-\e/2}\la^{2s-\e}
    +\kappa_{s}^{-1} \Lambda_{N,s}^{-1}
    \max\left\{1, \tfrac{N+2-2s}{4}\right\}+K_{N,s} \right)  \norm{W_\la(\cdot,t)}_{\mathcal{H}}
  \norm{\phi}_{\mathcal{H}},
\end{align}
for any $\la \in (0,\sqrt{T_0})$ and a.e. $t\in (0,1)$, thanks to \eqref{hp-g-subhomogeneous},
\eqref{ineq-trace}, \eqref{ineq-hardy-frac}, \eqref{def-h}  and the H\"older inequality.  Then, by
\eqref{ineq-hardy-frac} and \eqref{eq-prop-Wla-Wlat-bounded:4},
\begin{multline*}
  \left|\sideset{_{\mathcal{H}^*}}{_{\mathcal{H}}}{\mathop{\left\langle
          (W_\la)_t(\cdot,t), \phi\right\rangle}}\right| \le
  \Bigg(1+\mu\,\kappa_{s}^{-1} \Lambda_{N,s}^{-1}
  \max\left\{1, \tfrac{N+2-2s}{4}\right\}\\
  +C_gT_0^{\e/2}\left(K_{N,s}T_0^{\frac{2s-\e}{2}}+\kappa_{s}^{-1}
    \Lambda_{N,s}^{-1} \max\left\{1,
      \tfrac{N+2-2s}{4}\right\}+K_{N,s} \right) \Bigg)
  \frac{\norm{W_\la(\cdot,t)}_{\mathcal{H}}
    \norm{\phi}_{\mathcal{H}}}{t}.
\end{multline*}
Hence
\begin{equation*}
  \norm{(W_\la)_t(\cdot,t)}_{\mathcal{H}^*}\le
  \frac{\rm{const}}{t} \norm{W_\la(\cdot,t)}_{\mathcal{H}},
\end{equation*}
so that  \eqref{Wla-t-bounded} follows from \eqref{Wla-bounded}.

Finally, in view of \eqref{Wla-bounded} and \eqref{Wla-t-bounded}, we
can apply \cite[Corollary 8]{SJ} to obtain \eqref{Wla-compact}.
\end{proof}

\begin{proposition}\label{prop-blow-up}
  Let $V$ be a solution to \eqref{eq-weak-formulation-gaussian} such
  that $V\not\equiv0$ in $\R^{N+1}_+\times
(0,T)$  and
  let $\gamma$ be as in \eqref{limit-N}. Then $\gamma$ is an
  eigenvalue of problem
  \eqref{prob-eigenvalue-Ornstein-Uhlenbeck-operator}.  Furthermore,
  for any sequence $\la_n \to 0^+$, there exist a subsequence
  $\la_{n_k} \to 0^+$ and an eigenfunction $Y$ of problem
  \eqref{prob-eigenvalue-Ornstein-Uhlenbeck-operator} associated to
  $\gamma$ such that $\norm{Y}_{\mathcal{L}}=1$ and, for any
  $\tau \in (0,1)$,
\begin{align*}
  &\lim_{k \to \infty} \int_{\tau}^1\norm{\frac{V(z,\la_{n_k}^2t)}
    {\sqrt{H(\la_{n_k}^2)}}-t^\gamma  Y(z)}_{\mathcal{H}}^2 \, dt =0,\\
  &\lim_{k \to \infty} \sup\limits_{t \in [\tau,1]}
    \norm{\frac{V(z,\la_{n_k}^2t)}{\sqrt{H(\la_{n_k}^2)}}-t^\gamma Y(z)}_{\mathcal{L}} =0.
\end{align*}  
\end{proposition}

\begin{proof}
  Let $\la_n \to 0^+$. By Proposition
  \ref{prop-Wla-Wlat-bounded} and a diagonal argument, there exists a
  subsequence $\la_{n_k} \to 0^+$ and
  $\widetilde{W} \in \cap_{\tau \in (0,1)}
  (C^0([\tau,1],\mathcal{L})\cap L^2((\tau,1),\mathcal{H}))$ such that
  $\widetilde{W}_t \in \cap_{\tau \in (0,1)}
  L^2((\tau,1),\mathcal{H^*})$,
\begin{equation}\label{eq-prop-blow-up:1}
  W_{\la_{n_k}} \rightharpoonup \widetilde{W}
  \text{ weakly  in }  L^2((\tau,1),\mathcal{H}),
  \quad  (W_{\la_{n_k}})_t \rightharpoonup \widetilde{W}_t \text{ weakly in }  	L^2((\tau,1),\mathcal{H^*}),
\end{equation} 
and 
\begin{equation}\label{eq-prop-blow-up:2}
W_{\la_{n_k}} \to \widetilde{W} \text{ strongly in }  C^0([\tau,1],\mathcal{L}),
\end{equation}
as $k \to +\infty$, for any $\tau \in (0,1)$.
Since, by \eqref{def-W-la}, 
\begin{equation*}
\norm{W_{\la_{n_k}}(\cdot,1)}_{\mathcal{L}}=1,
\end{equation*}
from \eqref{eq-prop-blow-up:2} we obtain that
\begin{equation}\label{eq-prop-blow-up:3}
\norm{\widetilde{W}(\cdot,1)}_{\mathcal{L}}=1,
\end{equation}
hence $\widetilde{W} \not \equiv 0$.
Now we claim that 
\begin{equation}\label{eq-prop-blow-up:4}
W_{\la_{n_k}} \to \widetilde{W} \text{ strongly  in }  L^2((\tau,1),\mathcal{H})  \quad \text{ for any }\tau \in (0,1).
\end{equation}
Thanks to \eqref{eq-prop-Wla-Wlat-bounded:4} and
\eqref{eq-prop-blow-up:1}, we can pass to the limit in
\eqref{eq-weak-formulation-gaussian-Vla} thus obtaining, for any
$\phi \in \mathcal{H}$ and a.e. $t \in (0,1)$,
\begin{multline}\label{eq-prop-blow-up:5}
  \sideset{_{\mathcal{H}^*}}{_{\mathcal{H}}}
  {\mathop{\left\langle (\widetilde{W})_t(\cdot,t),\phi\right\rangle}}
  =\frac{1}{t}\int_{\R^{N+1}_+}y^{1-2s}\nabla \widetilde{W}(z,t)
  \cdot \nabla \phi(z) \, G(z) \, dz \\
  -\frac{1}{t}\int_{\R^N}\frac{\mu}{|x|^{2s}}
  \widetilde{w}(x,t)\Tr\phi(x)G(x,0)\, dx,
\end{multline}
where $\widetilde{w}(\cdot,t):=\Tr(\widetilde{W}(\cdot,t))$. Testing
the difference between \eqref{eq-weak-formulation-gaussian-Vla} and
\eqref{eq-prop-blow-up:5} with $(W_{\la_{n_k}}-\widetilde{W})$, 
integrating between $\tau$ and $1$, and taking into account
Remark
\ref{remark-parabolic-change-variables},  we obtain that
\begin{align*}
  \int_{\tau}^{1}&\left(\int_{\R^{N+1}_+}y^{1-2s}|\nabla \widetilde{W}-\nabla W_{\la_{n_k}}|^2 \, G \, dz\right) \, dt\\
 & \qquad-\int_{\tau}^{1}\left(\int_{\R^N}\frac{\mu}{|x|^{2s}} |\widetilde{w}(x,t)-w_{\la_{n_k}}(x,t)|^2 G(x,0)\, dx\right)\, dt\\
 & =\frac{1}{2}\norm{\widetilde{W}(\cdot,1)-W_{\la_{n_k}}(\cdot,1)}_{\mathcal{L}}^2-\frac{\tau}{2}\norm{\widetilde{W}(\cdot,\tau)-W_{\la_{n_k}}(\cdot, \tau)}_{\mathcal{L}}^2\\
 & \qquad-\frac{1}{2}\int_{\tau}^1\left(\int_{\R_+^{N+1}}y^{1-2s}|\widetilde{W}-W_{\la_{n_k}}|^2 G\, dz\right) \, dt \\
 & \qquad+\la^{2s}\int_{\tau}^{1}\left(\int_{\R^N}t^sh(\la\sqrt{t}x,\la^2t)w_{\la_{n_k}}(x,t)(w_{\la_{n_k}}(x,t)-\tilde{w}(x,t))G(x,0)\, dx\right) \, dt.\\
\end{align*}
Then by
\eqref{eq-prop-Wla-Wlat-bounded:4}, \eqref{eq-prop-blow-up:1} and
\eqref{eq-prop-blow-up:2} we conclude that
\begin{multline*}
  \lim_{k \to 0^+}
  \Bigg(\int_{\tau}^{1}\left(\int_{\R^{N+1}_+}y^{1-2s}
    \left(|\nabla \widetilde{W}-\nabla W_{\la_{n_k}}|^2
      +\frac12|\widetilde{W}-W_{\la_{n_k}}|^2\right)\, G \, dz\right) \, dt\\
  -\int_{\tau}^{1}\left(\int_{\R^N}\frac{\mu}{|x|^{2s}}
    |\widetilde{w}(x,t)-w_{\la_{n_k}}(x,t)|^2 G(x,0)\, dx\right)\,
  dt\Bigg)=0.
\end{multline*}
Thanks to Proposition \ref{prop-ineq-D+H} with $f\equiv 0$ and
  \eqref{eq-prop-blow-up:2},
we conclude that \eqref{eq-prop-blow-up:4} holds. Therefore, for any
$\tau \in (0,1)$,
\begin{equation}\label{eq-prop-blow-up:5.1}
\lim_{k \to \infty}\int_{\tau}^1 \norm{W_{\la_{n_k} }(\cdot,t)-\widetilde W(\cdot,t)}_{\mathcal{H}}^2 \, dt =0
\end{equation}
and, by \eqref{eq-prop-blow-up:2},
\begin{equation*}
\lim_{k \to \infty}\sup\limits_{t \in [\tau,1]} \norm{W_{\la_{n_k} }(\cdot,t)-\widetilde{W} (\cdot,t)}_{\mathcal{L}}^2 \, dt =0.
\end{equation*}
Let us define, for any $t \in (0,1)$,
\begin{align*}
&H_{\widetilde{W}}(t):=\int_{\R^{N+1}_+} y^{1-2s} \widetilde{W}^2G \, dz,\\
&D_{\widetilde{W}}(t):=\frac{1}{t}\left(\int_{\R^{N+1}_+} y^{1-2s} |\nabla \widetilde{W}|^2G \, dz - \int_{\R^N} \frac{\mu}{|x|^{2s}}\widetilde{w}^2(x,t) G(x,0)\, dx\right).
\end{align*}
Since $H_{\widetilde{W}}(1)=\|\widetilde{W}(\cdot,1)\|_{\mathcal{L}}^2=1$ by \eqref{eq-prop-blow-up:3}, then
\begin{equation*}
H_{\widetilde{W}}(t)>0 \text{ for any } t \in (0,1)
\end{equation*}
as we can prove arguing as in  Proposition
\ref{prop-H-positive}. Hence the function 
\begin{equation*}
\mathcal{N}_{\widetilde{W}}:(0,1)\to \R, \quad \mathcal{N}_{\widetilde{W}}(t):=\frac{tD_{\widetilde{W}}(t)}{H_{\widetilde{W}}(t)} 
\end{equation*}
is well defined.
Furthermore, from \eqref{def-N-scaled} and \eqref{def-W-la} , it follows that 
\begin{equation*}
\mathcal{N}_\la(t)=\frac{\int_{\R^{N+1}_+} y^{1-2s}|\nabla W_\la|^2 G \, dz-\int_{\R^N}  \left(\frac{\mu}{|x|^{2s}}w_\la^2(x,t) +\la^{2s}t^sh(\la \sqrt{t}x,\la^2t) w^2_\la(x,t)\right)G(x,0)\, dx}{\int_{\R^{N+1}_+} y^{1-2s}W_\la^2(z,t) G \, dz}.
\end{equation*}
Then, from  \eqref{ineq-hardy-frac},
\eqref{eq-prop-Wla-Wlat-bounded:4}, and \eqref{eq-prop-blow-up:5.1} 
we deduce that
\begin{equation*}
\lim_{k \to \infty}\mathcal{N}_{\la_{n_k}}(t)=\mathcal{N}_{\widetilde{W}}(t) \quad \text{ for a.e. }  t \in (0,1).
\end{equation*}
On the other hand,
$\mathcal{N}_{\la_{n_k}}(t)=\mathcal{N}(\la^2_{n_k}t)$ for any
$t \in (0,1)$ by \eqref{D-H-N-scaling} and hence
\begin{equation*}
  \mathcal{N}_{\widetilde W}(t)=\lim_{k \to \infty}\mathcal{N}_{\la_{n_k}}(t)=\lim_{k \to \infty}\mathcal{N}(\la_{n_k}^2t)=\gamma \quad \text{ for any  }  t \in (0,1),
\end{equation*}
with $\gamma$ as in \eqref{limit-N}. It follows that
$\mathcal{N}'_{\widetilde{W}}(t) \equiv 0$ in $(0,1)$. In view of
Proposition \ref{prop-N-W11} in the case $h \equiv 0$, we deduce that
\begin{equation*}
\left(\int_{\R_+^{N+1}}y^{1-2s}\widetilde{W}_t^2G\, dz\right)\left( \int_{\R^{N+1}_+}y^{1-2s}\widetilde{W}^2 G \, dz\right)
=\left(\int_{\R_+^{N+1}}y^{1-2s} \widetilde{W}_t\widetilde{W}Gdz\right)^2
\end{equation*}
for a.e. $t \in (0,1)$. In particular, for the vectors
$\widetilde{W}_t(\cdot,t)$ and $\widetilde{W}(\cdot,t)$ in
$\mathcal{L}$, equality holds in the Cauchy-Schwarz inequality for
a.e. $t\in(0,1)$.
Hence there exists a function $\beta:(0,1) \to \R$ such that
\begin{equation}\label{eq-prop-blow-up:6.2}
 \widetilde{W}_t(z,t)=\beta(t)\widetilde{W}(z,t) \quad \text{ for a.e. } z \in \R^{N+1}_+ \text{ and  a.e. }   t \in (0,1).
\end{equation}
Thanks to \eqref{eq-H'} and Remark \ref{remark-Lt-Ht*},
\begin{equation*}
D_{\widetilde{W}}(t)=\sideset{_{\mathcal{H}^*}}{_{\mathcal{H}}}{\mathop{\left\langle (\widetilde{W})_t(\cdot,t),\widetilde{W}(\cdot,t)\right\rangle}}=\beta(t)H_{\widetilde{W}}(t),
\end{equation*}
and so 
\begin{equation}\label{eq-prop-blow-up:6.2.2}
\beta(t)=\frac{\gamma}{t} \quad \text{ for a.e. } t \in (0,1).
\end{equation}
Combining \eqref{eq-prop-blow-up:5}, \eqref{eq-prop-blow-up:6.2} and
\eqref{eq-prop-blow-up:6.2.2}, we conclude that  $\widetilde{W}$ satisfies
\begin{equation}\label{eq-prop-blow-up:6.3}
\gamma \int_{\R^{N+1}_+} y^{1-2s}\widetilde{W}\phi \,G\, dz=\int_{\R^{N+1}_+}y^{1-2s}\nabla \widetilde{W}\cdot \nabla \phi \, G \, dz 
-\int_{\R^N}\frac{\mu}{|x|^{2s}} \widetilde{w}\phi \,G(x,0)\, dx,
\end{equation}
for all $\phi \in \mathcal{H}$ and a.e. $t \in (0,1)$.

Furthermore from \eqref{eq-prop-blow-up:6.2} it follows that, letting
$\widetilde{W}^\eta(z,t):=\widetilde{W}(z, \eta^2 t)$ for any
$\eta >0$,
\begin{equation*}
\frac{d\widetilde{W}^\eta}{d \eta}(z,t)=\frac{2 \gamma}{\eta} \widetilde{W}^\eta(z,t) \quad \text{ in a distributional sense and    a.e. in }  \R^{N+1}_+ \times (0,1).
\end{equation*}
An integration yields 
\begin{equation*}
\widetilde{W}^\eta(z,t)=\eta^{2\gamma}\widetilde{W}(z,t) \quad \text{ for all } \eta >0  \text{ and  a.e. in }  \R^{N+1}_+ \times (0,1).
\end{equation*} 
Let $Y(z):=\widetilde{W}(z,1)$. Then $\norm{Y}_\mathcal{L}=1$ and
\begin{equation}\label{eq-prop-blow-up:8}
\widetilde{W}(z,t)=\widetilde{W}^{\sqrt{t}}\left(z,1\right)=t^\gamma Y\left(z\right) \text{ for  a.e. } z \in \R^{N+1}_+ \text { and a.e. } t \in (0,1).
\end{equation}
Moreover, from \eqref{eq-prop-blow-up:6.3} and
\eqref{eq-prop-blow-up:8}, $Y\in \mathcal{H}$ and $Y$ satisfies
\begin{equation*}
  \gamma\int_{\R^{N+1}_+} y^{1-2s}Y\phi G\,
  dz=\int_{\R^{N+1}_+}y^{1-2s}\nabla Y
  \cdot \nabla \phi \, G \, dz \\
  -\int_{\R^N}\frac{\mu}{|x|^{2s}} \Tr(Y)\Tr(\phi)G(x,0)\, dx
\end{equation*}
for any $\phi\in C_c^{\infty}(\overline{\R^{N+1}_+})$, i.e.
$\gamma$ is an eigenvalue of problem
\eqref{prob-eigenvalue-Ornstein-Uhlenbeck-operator} and $Y$ is an
associated eigenfunction. The proof is then complete.
\end{proof}

Now we study  the asymptotic behavior of  $H(t)$ as $t \to 0^+$.
\begin{proposition}\label{prop-limit-H}
  Let $\gamma$ be as in \eqref{limit-N}. Then the limit
  $\lim_{t \to 0^+} t^{-2\gamma} H(t)$ exists and it is finite.
\end{proposition}

\begin{proof}
  Thanks to \eqref{ineq-H-estimates-above}, we only need to show that
  the limit exists. By \eqref{eq-H'}, Proposition
  \ref{prop-N-W11}, and \eqref{limit-N},
\begin{align*}
  \frac{d}{dt}(t^{-2\gamma}H(t))&
                                  =-2\gamma t^{-2 \gamma -1}H(t)   +t^{-2\gamma}H'(t)=2t^{-2\gamma-1}(tD(t)-\gamma H(t))\\
                                &=2t^{-2\gamma-1}H(t)\int_0^t(\nu_1(\tau)+\nu_2(\tau)) \, d\tau
\end{align*}
for a.e.  $t \in (0,T_0)$, where $\nu_1$ and $\nu_2$ have been
defined in \eqref{def-nu12}.  An integration over $(t,T_0)$ yields
\begin{equation*}
  \frac{H(T_0)}{T_0^{2\gamma}}-\frac{H(t)}{t^{2\gamma}}
  =\int_{t}^{T_0} 2\rho^{-2\gamma-1}H(\rho)
  \left(\int_0^\rho\nu_1(\tau)d \tau \right)  d\rho
  +\int_{t}^{T_0} 2\rho^{-2\gamma-1}H(\rho)
  \left(\int_0^\rho\nu_2(\tau)d \tau \right)  d\rho.
\end{equation*}
Since by Proposition \ref{prop-N-W11} it follows that $\nu_1 \ge 0$, then the function 
\begin{equation*}
  t \to \int_{t}^{T_0} 2\rho^{-2\gamma-1}
  H(\rho)\left(\int_0^\rho\nu_1(\tau)d \tau \right) \, d\rho
\end{equation*}
is non-increasing on $(0,T_0)$ and it has a limit as
  $t\to0^+$.  From \eqref{hp-g-Lr},  \eqref{def-h}, Proposition
\ref{prop-nu2-estimates} and Proposition \ref{prop-N-bounded} we
deduce that
\begin{equation*}
\left|\int_0^\rho\nu_2(\tau)d\tau \right| \le {\rm{const}}\, \rho^\delta,
\end{equation*}
where $\delta:=\min\{\frac{\e}{2},1-\frac{1}{r}\}$.
Then by \eqref{ineq-H-estimates-above}
\begin{equation*}
\left| 2\rho^{-2\gamma-1}H(\rho)\left(\int_0^\rho\nu_2(\tau)d \tau \right) \, d\rho\right| \le {\rm{const}}\, \rho^{-1+\delta}
\end{equation*}
hence the function 
\begin{equation*}
\rho \to 2\rho^{-2\gamma-1}H(\rho)\left(\int_0^\rho\nu_2(\tau)d \tau \right) \, dr 
\end{equation*}
belongs to $L^1(0,T_0)$. We conclude that limit
$\lim_{t \to 0^+} t^{-2\gamma} H(t)$ exists.
\end{proof}

\begin{proposition}\label{prop-limit-H-positive}
Let $\gamma$ be as in \eqref{limit-N}. Then 
$\lim_{t \to 0^+} t^{-2\gamma} H(t)>0$.
\end{proposition}

\begin{proof}
We argue by contradiction assuming that $\lim_{t \to 0^+} t^{-2\gamma} H(t)=0$. 

Since $V_\la(z,1)=V(z, \la^2) \in \mathcal{H}$ and
$h(\la x, \la^2 )\Tr(V_\la)(x,1) \in \mathcal{H}^*$ for a.e.
$\la \in (0,\sqrt{T_0})$, by Proposition \ref{prop-eigenvalues} we can
expand them in $\mathcal{L}$ and $\mathcal{H}^*$ respectively as
\begin{align*}
  V_\la(z,1)&=\sum_{(n,j)\in \mathbb{N} \times(\mathbb{N}\setminus
    \{0\})}
    V_{n,j}(\la)\widetilde{Y}_{n,j}(z) \quad \text{ in } \mathcal{L},\\
  h(\la x, \la^2 )\Tr(V_\la)(x,1)&=
    \sum_{(n,j)\in \mathbb{N} \times(\mathbb{N}\setminus
    \{0\})}\xi_{n,j}(\la)\widetilde{Y}_{n,j}(z)
    \quad \text{ in } \mathcal{H}^*,
\end{align*} 
where 
\begin{align}
  &V_{n,j}(\la):=\int_{\R^{N+1}_+} y^{1-2s}V_\la(z,1)\widetilde{Y}_{n,j}(z)G(z)\, dz, \label{def-V-la-mj}\\
  &\xi_{n,j}(\la):=\!\!\sideset{_{\mathcal{H}^*}}{_{\mathcal{H}}}{\mathop{\left\langle h(\la \cdot, \la^2 )\Tr(V_\la)(\cdot,1), \widetilde{Y}_{n,j}\right\rangle}}\!\!\label{def-xi-la-mj}
  =\!\int_{\R^N}\! h(\la 
    x,\la^2)v_\la(x,1)\Tr(\widetilde{Y }_{n,j}) G(x,0)\, dx, 
\end{align}
for a.e. $\la
\in (0,\sqrt{T_0})$. By Parseval's identity
\begin{equation*}
H(\la^2)=\sum_{(m,i)\in \mathbb{N} \times(\mathbb{N}\setminus
  \{0\})}(V_{m,i}(\la))^2 \ge (V_{n,j}(\la))^2 
\end{equation*}
for any $\la \in (0,T_0)$ and $(n,j)\in \mathbb{N} \times(\mathbb{N}\setminus
  \{0\})$. Hence, from $\lim_{t \to 0^+} t^{-2\gamma} H(t)=0$ we deduce that 
\begin{equation}\label{eq-prop-limit-H-positive:1}
  \lim_{\la \to 0^+} \la^{-2\gamma}V_{n,j}(\la)=0 \quad \text{ for any }
  (n,j)
  \in \mathbb{N} \times(\mathbb{N}\setminus \{0\}).
\end{equation}
By \eqref{def-V-la}, Remarks \ref{remark-parabolic-change-variables},
  \ref{remark-Lt-Ht*}, and Proposition \ref{prop-Vt-L2}, it is
easy to see that, for any
$(n,j)\in \mathbb{N} \times(\mathbb{N}\setminus \{0\})$,
$V_{n,j}$ is absolute continuous on any closed sub-interval of
$(0,\sqrt{T_0})$ and 
\begin{equation*}
  \frac{d}{d \la } V_{n,j}(\la)=
  \df{\mathcal{H}^*}{\frac{d}{d \la }
    V_\la(\cdot,1)}{\widetilde{Y}_{n,j}}{\mathcal{H}}
  =\ps{\mathcal{L}}{\frac{d}{d \la } V_\la(\cdot,1)}{\widetilde{Y}_{n,j}}
  =2 \la \ps{\mathcal{L}}{ V_t(\cdot,\la^2)}{\widetilde{Y}_{n,j}}
\end{equation*}
a.e. and in a distributional sense in $(0,\sqrt{T_0})$.
Furthermore, for any $(n,j) \in \mathbb{N} \times(\mathbb{N}\setminus \{0\})$, since $\widetilde{Y}_{n,j}$ is 
an eigenfunction of problem
\eqref{prob-eigenvalue-Ornstein-Uhlenbeck-operator} we have that
\begin{align*} 
  &2\la\int_{\R^{N+1}_+}y^{1-2s}V_t( z,\la^2)\widetilde{Y}_{n,j}(z)G(z) \, dz
    =\frac{2}{\la}\Bigg(\int_{\R^{N+1}_+}y^{1-2s}\nabla V(\cdot,\la^2)\cdot \nabla \widetilde{Y}_{n,j} \, 
    G \, dz \\ 
  &\quad\quad-\int_{\R^N}\left(\frac{\mu}{|x|^{2s}} v(x,\la^2)\Tr(\widetilde{Y}_{n,j})+\la^{2s}h(\la 
    x,\la^2)v(x,\la^2)\Tr(\widetilde{Y}_{n,j})\right)G(x,0)\, dx \Bigg)\\
  &=\frac{2}{\la}\gamma_{n,j}\int_{\R^{N+1}_+}y^{1-2s}V(\cdot,\la^2) \widetilde{Y}_{n,j} G \, dz \\
  &\quad\quad-2\la^{2s-1}\int_{\R^N} h(\la 
    x,\la^2)v(x,\la^2)\Tr(\widetilde{Y}_{n,j}) G(x,0)\, dx
    =\frac{2}{\la}\gamma_{n,j}V_{n,j}(\la)-2 \la^{2s-1}\xi_{n,j}(\la), 
\end{align*}
for a.e. $\la \in (0,\sqrt{T_0})$, by
\eqref{eq-weak-formulation-gaussian}, \eqref{def-V-la-mj} and
\eqref{def-xi-la-mj}.  In conclusion we have proved that, for any
$(n,j) \in\mathbb{N} \times(\mathbb{N}\setminus \{0\})$,
\begin{equation*}
\frac{d}{d \la } V_{n,j}(\la)=\frac{2}{\la}\gamma_{n,j}V_{n,j}(\la)-2 \la^{2s-1}\xi_{n,j}(\la)
\end{equation*}
for a.e. $\la \in (0,\sqrt{T_0})$ and  in a distributional sense.
An integration  yields
\begin{equation} \label{eq-prop-limit-H-positive:2}
  V_{n,j}(\bar \la)={\bar
    \la}^{2\gamma_{n,j}}\left(\la^{-2\gamma_{n,j}}V_{n,j}(\la)+2\int_{\bar
      \la}^\la \tau^{2s-1-2\gamma_{n,j}}\xi_{n,j}(\tau)\, d \tau
  \right)
\end{equation}
for any $\bar \la , \la \in (0,\sqrt{T_0})$.

Thanks to Proposition \ref{prop-blow-up}, there exists an
eigenvalue $\gamma_{m_0,k_0}$ of
\eqref{prob-eigenvalue-Ornstein-Uhlenbeck-operator} such that
$\gamma=\gamma_{m_0,k_0}$.  Then, for any $(n,j) \in J_0$ (see
\eqref{def-J0}), we can estimate $\xi_{n,j}$ as follows. From the
H\"older inequality, the fact that
$\widetilde{Y}_{n,j} \in \mathcal{H}$, \eqref{ineq-D+H}, and
\eqref{ineq-h} it follows that
\begin{align}\label{ineq-xi}
&\la^{2s}|\xi_{n,j}(\la)|= \la^{2s}\left|\int_{\R^N}h(\la x,\la^2) v( x,\la^2) \Tr(\widetilde{Y}_{n,j})(x) G(x,0) \, dx\right| \\
\notag&\le\left(\int_{\R^N}\!\la^{2s}|h(\la x,\la^2)| |v( x,\la^2)|^2  G(x,0) dx\!\right)^{\!\!\frac{1}{2}}\!\!
\left(\int_{\R^N}\!\la^{2s}|h(\la x,\la^2)|
        |\Tr(\widetilde{Y}_{n,j})|^2 G(x,0) 
        dx\!\right)^{\!\!\frac{1}{2}}\\
  &\notag\le \rm{const} \,  \la^{\e +2 \gamma}
\end{align}
for any $\la \in (0,\sqrt{T_0})$, where Proposition
\ref{prop-N-bounded} and \eqref{ineq-H-estimates-above} have been used. 

It follows that $\tau \to \tau^{2s-1-2\gamma}\xi_{n,j}(\tau)$ belongs
to $L^1(0,\sqrt{T_0})$ for any $(n,j) \in J_0$.  Passing to the limit
as $\bar \la \to 0^+$ in \eqref{eq-prop-limit-H-positive:2}, from
\eqref{eq-prop-limit-H-positive:1} we deduce that
\begin{equation}\label{eq-prop-limit-H-positive:3}
V_{n,j}(\la)=-2\la^{2\gamma}\int_{0}^{\la}  \tau^{2s-1-2\gamma}\xi_{n,j}(\tau)\, d \tau \quad \text{ for any } (n,j) \in J_0.
\end{equation}
Combining \eqref{ineq-xi} and \eqref{eq-prop-limit-H-positive:3} we
obtain that
\begin{equation*}
  |V_{n,j}(\la)| \le \mathop{\rm const}
  \la^{\e+2 \gamma} \quad \text{ for any } \la
  \in  (0,\sqrt{T_0}) \text{ and some const}>0 \text{ independent of } \la.
\end{equation*}
Fixing some $\sigma \in (0,\e)$, by \eqref{ineq-H-estimates-below}
there exists a constant $K(\sigma)>0$
such that, for any $\la \in (0,\sqrt{T_0})$,
\begin{equation*}
H(\la^2)\ge K(\sigma) \la^{2(2\gamma +\sigma)}.
\end{equation*}
We conclude that 
\begin{equation}\label{eq-prop-limit-H-positive:4.1}
\frac{|V_{n,j}(\la)|}{\sqrt{H(\la^2)}} = O\left(\la^{\e-\sigma}\right)=o(1) \quad \text{ as } \la \to 0^+.
\end{equation}
On the other hand, for any sequence $\la_i\to 0^+$,
Proposition \ref{prop-blow-up} provides a subsequence
$\la_{i_k} \to 0^+$ and an eigenfunction $Y$ of
\eqref{prob-eigenvalue-Ornstein-Uhlenbeck-operator} associated to the
eigenvalue $\gamma$ such that
\begin{equation*}
  \frac{V_{\la_{i_k}}(\cdot,1)}{\sqrt{H(\la_{i_k}^2)}} \to Y
  \quad \text{ strongly in } \mathcal{L} \text{ as } k \to \infty.
\end{equation*}
In particular, for any $(n,j) \in J_0$,
\begin{equation}\label{eq-prop-limit-H-positive:5}
  \frac{V_{n,j}(\la_{i_k})}{\sqrt{H(\la_{i_k}^2)}}
  =
  \ps{\mathcal{L}}{\frac{V_{\la_{i_k}}(\cdot,1)}{\sqrt{H(\la_{i_k}^2)}}}{\widetilde{Y}_{n,j}}
  \to\ps{\mathcal{L}}{Y}{\widetilde{Y}_{n,j}}\quad \text{ as } k \to \infty.
\end{equation}
From \eqref{eq-prop-limit-H-positive:4.1} and
\eqref{eq-prop-limit-H-positive:5} we deduce that
$(Y,\widetilde{Y}_{n,j})_{\mathcal{L}}=0$ for any $(n,j) \in J_0$.
We conclude that $Y \equiv 0$, a contradiction.
\end{proof}

\begin{proof}[\textbf{ Proof of Theorems \ref{theorem-blow-up-extended} and \ref{theorem-blow-up}}]
  In view of Proposition \ref{prop-blow-up}, there exists an
  eigenvalue $\gamma_{m_0,k_0}$ of problem
  \eqref{prob-eigenvalue-Ornstein-Uhlenbeck-operator} such that
  \eqref{limit-N-forward} holds. Let $J_0$ be as in \eqref{def-J0} and
  $\la_i \to 0^+$ as $i \to +\infty$. Thanks to Proposition
  \ref{prop-blow-up} and Proposition \ref{prop-limit-H-positive} there
  exists a subsequence $\{\la_{i_k}\}_{k \in \mathbb{N}}$ and real
  numbers $\{\beta_{n,j}:(n,j )\in J_0\}$ such that
  $\beta_{\tilde n, \tilde j}\neq 0$ for some
  $(\tilde n,\tilde j )\in J_0$ and, for any $\tau \in (0,1)$,
\begin{equation}\label{eq-proof-theorem-blow-up-extended:1}
\lim_{k \to \infty}\int_{\tau}^{1}\norm{\la_{i_k}^{-2\gamma_{m_0,k_0}}V( z,\la_{i_k}^2 t)-t^{\gamma_{m_0,k_0}}\sum_{(n,j) \in J_0}\beta_{n,j}\widetilde{Y}_{n,j}(z)}^2_{\mathcal{H}} dt=0
\end{equation}
and 
\begin{equation}\label{eq-proof-theorem-blow-up-extended:2}
\lim_{k \to \infty}\sup_{t \in [\tau,1]}\norm{\la_{i_k}^{-2\gamma_{m_0,k_0}}V( z,\la_{i_k}^2 t)-t^{\gamma_{m_0,k_0}}\sum_{(n,j) \in J_0}\beta_{n,j}\widetilde{Y}_{n,j}(z )}^2_{\mathcal{L}}=0.
\end{equation}
It follows that 
\begin{equation}\label{eq-proof-theorem-blow-up-extended:2.1}
  \la_{i_k}^{-2\gamma_{m_0,k_0}}V( z,\la_{i_k}^2 )\to \sum_{(n,j) \in
    J_0}\beta_{n,j}\widetilde{Y}_{n,j}(z)
  \quad \text{ strongly in } \mathcal{L} \text{ as } k \to \infty.
\end{equation}
Let us prove that $\{\beta_{n,j}:(n,j )\in J_0\}$ depends neither on the
sequence $\{\la_{i}\}_{i \in \mathbb{N}}$ nor on its subsequence
$\{\la_{i_k}\}_{k \in \mathbb{N}}$.  Let $\Lambda \in (0,\sqrt{T_0})$
and let $V_{n,j}$, $\xi_{n,j}$ be as in \eqref{def-V-la-mj} and
\eqref{def-xi-la-mj} respectively.  From
\eqref{eq-proof-theorem-blow-up-extended:2.1} we obtain that, for any
$(n,j) \in J_0$,
\begin{equation*}
  \la_{i_k}^{-2\gamma_{m_0,k_0}}V_{n,j}(\la_{i_k})\to\beta_{n,j}\quad\text{as $k \to \infty$}.
\end{equation*}
By \eqref{eq-prop-limit-H-positive:2},  for any $(n,j) \in J_0$ and $\la\in (  0,\Lambda)$,
\begin{equation*}
  V_{n,j}(\la)={\la}^{2\gamma_{m_0,k_0}}
  \left(\Lambda^{-2\gamma_{m_0,k_0}}V_{n,j}(\Lambda)
    +2\int_{\la}^\Lambda  \tau^{2s-1-2\gamma_{m_0,j_0}}\xi_{n,j}(\tau)\, d \tau \right).
\end{equation*}
Furthermore, proceeding as in Proposition
\ref{prop-limit-H-positive}, we can prove that 
$\tau \to \tau^{2s-1-2\gamma_{m_0,j_0}}\xi_{n,j}(\tau)$ belongs to
$L^1(0,\sqrt{T_0})$. Hence
\begin{align*}
  \beta_{n,j}&=\Lambda^{-2\gamma_{m_0,k_0}}V_{n,j}(\Lambda)+2\int_0^\Lambda  \tau^{2s-1-2\gamma_{m_0,j_0}}\xi_{n,j}(\tau)\, d \tau\\
             &=\Lambda^{-2\gamma_{m_0,k_0}}\int_{\R^{N+1}_+} y^{1-2s}V( z,\Lambda^2)\widetilde{Y}_{n,j}(z) G(z)\, dz \\
             &\qquad+2\int_0^\Lambda\tau^{2s-1-2\gamma_{m_0,k_0}}\left(\int_{\R^N} h(\tau x,\tau^2)v( x,\tau^2)\Tr(\widetilde{Y}_{n,j})(x) G(x,0)\, dx \right)\, d\tau,
\end{align*}
so that $\beta_{n,j}$ depends neither on the sequence
$\{\la_{i}\}_{i \in \mathbb{N}}$ nor on its subsequence
$\{\la_{i_k}\}_{k \in \mathbb{N}}$ for any $(n,j) \in J_0$. Then, by
the Urysohn subsequence principle, we conclude that the convergences in
\eqref{eq-proof-theorem-blow-up-extended:1} and
\eqref{eq-proof-theorem-blow-up-extended:2} actually hold as
$\la \to 0^+$, thus proving Theorem \ref{theorem-blow-up-extended}.
Theorem \ref{theorem-blow-up} follows from Theorem
\ref{theorem-blow-up-extended} and the continuity of
the trace operator $\Tr$ from $\mathcal{H}$ into $L^2(\R^N,G(x,0))$,
see Proposition \ref{prop-trace-ineq}.
\end{proof}

The strong unique continuation principles stated in Corollaries
\ref{corollary-unique-principle-extended} and
\ref{corollary-unique-principle} easily follow from Theorem
\ref{theorem-blow-up-extended} and Theorem \ref{theorem-blow-up}.
\begin{proof}[\textbf{ Proof of Corollaries
    \ref{corollary-unique-principle-extended}
    and \ref{corollary-unique-principle}}]
  We start by proving Corollary
  \ref{corollary-unique-principle-extended}.  Let us assume by
  contradiction that $W \not \equiv 0$ on $\R^{N+1}_+ \times(-T,0)$
  and let $\gamma_{m_0,k_0}$ be as in Theorem
  \ref{theorem-blow-up-extended}.  In view of
  \eqref{eq-W-O} we have that
\begin{equation*}
  \lim_{\la \to 0^+} \la^{-2\gamma_{m_0,k_0}}t^{-\gamma_{m_0,k_0}}W(\la
  z\sqrt{t},-\la^2 t)=0
  \quad \text{ for a.e. } (z,t) \in \R^{N+1}_+\times (0,1).
\end{equation*}
On the other hand, by Theorem \ref{theorem-blow-up-extended} there
exists $Y \in \mathcal{H} \setminus \{0\}$ such that $Y$ is an
eigenfunction of problem
\eqref{prob-eigenvalue-Ornstein-Uhlenbeck-operator} and, for a.e.
$z \in \R^{N+1}_+$ and $t \in (0,1)$, 
\begin{equation*}
\lim_{n \to \infty} \la_n^{-2\gamma_{m_0,k_0}}t^{-\gamma_{m_0,k_0}}W(\la_n z \sqrt{t},-\la_n^2 t)= Y(z),
\end{equation*}
along a sequence $\lambda_n\to0^+$.
We conclude that $Y \equiv 0$, thus reaching a contradiction. In the
same way, we can deduce Corollary \ref{corollary-unique-principle}
from Theorem \ref{theorem-blow-up}, in
view of Proposition \ref{prop-trace-Y-0}.
\end{proof}

\end{document}